\documentclass[a4paper]{article}

\usepackage{fullpage} 
\usepackage{parskip} 
\usepackage{tikz} 
\usepackage{amsmath}
\usepackage{hyperref}
\usepackage{xcolor}
\usepackage{soul}
\usepackage{amsmath,amssymb}
\DeclareMathOperator*{\argmin}{arg\,min}

\usepackage{amsthm}
\usepackage{algorithm}
\usepackage{subcaption}
\usepackage{algpseudocode}
\newtheorem{theorem}{Theorem}[section]
\newtheorem{definition}[theorem]{Definition}
\newtheorem{lemma}[theorem]{Lemma}
\newtheorem{proposition}[theorem]{Proposition}

\newtheorem{assumption}{Assumption} 
\usepackage{multirow}
\usepackage{comment}

\title{Pseudoconvex Problems in Operational Decision Systems: Algorithms for Joint Learning and Optimization}
\author{Zijun Li and Aswin Kannan\footnote{
The first author is reachable at zijun.li@hu-berlin.de. The first author is a PhD candidate at Humboldt Universitaet zu Berlin, Germany under the guidance of the second author. The second author is the corresponding author and is reachable at both aswin.kannan@hu-berlin.de and aswin.kannan@iiitb.ac.in. The ORCID-ID for the corresponding author is 0000-0002-4698-7678. \\ 
The second author was a faculty member and Junior Research Group Leader at Humboldt Universitaet zu Berlin till December 2024. He is currently a faculty member at the International Institute of Information Technology, Bangalore, India (from January 2025). He still retains an affiliation to his Junior Group in Berlin.}}

\begin{document}
\maketitle
\abstract{
We consider joint optimization and learning problems arising in real-time decision systems.
While most existing work focuses primarily on convex, revenue-based objectives, we extend this line of research to multi-objective formulations. In energy systems, for instance, we incorporate metrics such as renewable penetration and generation costs. Our key focus, however, is on a class of problems with a pseudoconvex structure—a natural relaxation of convexity. Representative examples include fractional objectives in energy management and logit-based revenue models in retail. The outer-level problem optimizes these pseudoconvex objectives, while the inner-level problem involves training a machine learning model using historical data. Our contributions are twofold. First, we propose a simultaneous learning-and-optimization framework that iteratively updates both inner- and outer-level variables. Second, we develop convergent algorithms for these problem classes under realistic mathematical assumptions. Using real-world datasets, we evaluate the computational performance of our methods and highlight an important observation: there exist clear trade-offs between inexact learning and computational time when assessing final solution quality.
}

\section{Introduction}
Learning and optimization are often deeply interconnected, with learning traditionally performed first and optimization carried out thereafter. Consider retail pricing: revenue is given by the product of price and sales, but sales responses to price rarely admit closed-form expressions, necessitating the use of learned models fitted to historical data. A similar paradigm arises in energy management, where fuel-cost minimization and revenue maximization rely on learned relationships derived from historical commodity prices across multiple energy sources. When multiple objectives—such as cost, emissions, and risk—must be simultaneously balanced, the resulting problem can be expressed as
\begin{align}
\min_{x,\theta} \hspace{1mm} F(x,\theta) &= \left(f_{1}(x, \theta^{*}),..., f_{m}(x, \theta^{*}) \right), \quad
\text{where} \quad \theta^* = \argmin_{\theta \in \Theta} \hspace{1mm} h(\theta, y^{\text{fea}}, y^{\text{true}}).
\label{eq:primform}
\end{align}
Here, $y^{\text{fea}}$ and $y^{\text{true}}$ denote features and labels, respectively, while $\theta$ represents model parameters such as coefficients in regression models or neural-network weights. The decision variables $x$ encode operational or planning actions—fuel dispatch levels in energy systems or product prices in retail—while $F$ captures the optimization objectives, and $h$ is the learning objective. Examples include emissions or cost in $F$, and accuracy, sparsity, or computational bias in $h$. 

The conventional approach solves the learning problem in full before addressing the outer optimization. While intuitive, this separation can be computationally expensive and is often unnecessary when a high-quality solution, but suboptimal solution suffices. Motivated by this, recent research has examined simultaneous learning and optimization, where learning and decision-making interleave and inform each other throughout the process~\cite{kannan2015distributed,ahmadi25learning,AhmadiShanbhag2020,aybat22learning,Elmachtoub2022}. However, these works largely remain within convex regimes, without relaxing structural assumptions at either level. More recent frontier works explore certain weaknesses~\cite {Mandi_2024,kotary2024learningjointmodelsprediction}, though these advances primarily emphasize computational aspects. The survey by Mandi et al.~\cite{Mandi_2024} provides a detailed account of this emerging landscape. Some works~\cite{Schulman2021,BanRudin2019,BertsimasKallus2020} have considered solving learning problems to an approximate level and then performing the outer optimization completely with this error. These approaches are application-specific, and most of them also fall under the umbrella of convex optimization and single-objective formulations. 


In contrast, our work introduces a framework that accommodates multiple objectives and extends classical assumptions to pseudoconvex outer-level problems, which naturally arise in energy systems and retail management. These settings have received limited attention in existing literature, leaving important gaps. We highlight these gaps and articulate our contributions in the following sections. We also note that the proposed framework is flexible and can extend to broader applications and generalizations—an agenda we view as promising for future research.

\subsection{Gaps in Literature}
Before delving into the main applications, we first provide a formal definition of pseudoconvexity. Although pseudoconvexity represents a relaxation of convexity, it retains several computationally tractable properties. While we discuss these properties in more detail in the following sections, we present its formal definition below.

\begin{definition}
Consider a differentiable function $f_i:\Theta \subset \mathbb{R}^{n} \rightarrow \mathbb{R}$. Let $X$ be a convex set. Then, $f_i$ is pseudoconvex if the following holds:
\begin{align*}
\forall x_a, x_b \in X:\qquad \nabla f(x_a)^T (x_b - x_a) \geq 0 \;\Longrightarrow\; f(x_b) \geq f(x_a). 
\end{align*}
\end{definition}

In addition to our primary applications, pseudoconvex problems arise frequently in ratio optimization---for instance, in lift-to-drag optimization in aircraft design and various other engineering contexts. They are also relevant in modeling competitive economic equilibria, particularly in pricing and strategic decision-making settings.

\paragraph{Retail Management:}
From the retailer's perspective, the central objective is to set prices that maximize revenue. In the retail literature, much of the prior work has focused on revenue maximization subject to constraints such as promotions, markdowns, inventory levels, and inter-item dependencies. Demand can take several forms, including logit~\cite{mcfadden84}, probit~\cite{probit2001}, and multiplicative competitive interaction models~\cite{mci-15,retail-best-14}, with a comprehensive survey provided in~\cite{demand-survey13}. Many existing studies rely on closed-form demand expressions, whereas in practice these expressions are typically unknown and must be inferred from data. Large online retailers such as Amazon, for example, continuously update demand estimates through online learning.

In our work, we approximate demand using the logit model, which renders the inner problem a regression task rather than requiring complex models such as decision trees or neural networks. The logit model is well established in the literature, and readers may refer to~\cite{mcfadden84, logit23, Kreimeier2023} for further details. For all subsequent discussions, including the OPF problem, the structural formulation follows~(\ref{eq:primform}). In the retail setting, examples of objective functions include revenue, consumer satisfaction, profit, and segment-wise demand. The inner problem corresponds to estimating the coefficients of the logit or other demand model.

Although tree-based models and neural networks can also be used to approximate demand, such approaches are less common in the retail literature~\cite{demand-ml-22}. We consider such models primarily in the context of energy systems, where they are more standard. While the generic multiobjective formulation applies to retail, we restrict ourselves to single-objective formulations in this paper, as the theoretical results we develop pertain specifically to single-objective settings. A natural extension involves portfolio management, where the decision variable represents the fraction of a portfolio invested in an asset. Such formulations also give rise to fractional objectives and pseudoconvex problems~\cite{subramanian2010fractional, kannan2019optimal}. Although we do not pursue this direction here, our framework extends to this setting without loss of generality.

\paragraph{Energy Systems:}
In the context of energy systems, two primary challenges frequently arise: the first concerns DC Optimal Power Flow (OPF) for direct-current-based systems, and the second involves renewable energy prediction. Both problems lead to formulations similar to~(\ref{eq:primform}). Renewable energy prediction itself is a natural extension of OPF and unit commitment problems. The primary objective of OPF is to generate power at minimal cost while satisfying constraints related to environmental impact, ramping capabilities, and network supply--demand balance. Extensive research has explored both single- and multiobjective versions of this problem. However, these studies typically assume that cost coefficients and generation levels are either expressed in closed form or are known in advance. In practice, this assumption becomes increasingly unrealistic as renewable penetration grows. These parameters are often estimated from historical data, requiring the solution of a machine learning problem. At a high level, this process follows the structure described in~(\ref{eq:primform}). As illustrative examples, in this formulation $f_1$ may represent oil generation cost, $f_2$ renewable energy penetration, $f_3$ carbon emissions, and so forth. The function $h(\cdot)$ denotes the loss function used in the ML model (e.g., MSE), $x$ represents the oil generation level, $\theta$ the model hyperparameters to be optimized, and $y^{\text{fea}}$ and $y^{\text{true}}$ the features and labels obtained from historical data.

While much of the existing literature on pseudoconvex formulations focuses on solving the inner problem $h(\cdot)$ to full optimality~\cite{suneja2011optimality, kleinert2021survey, ji2021bilevel}, our work adopts a more integrated perspective in which learning and optimization are coupled and the inner problem may be solved inexactly. This reflects practical large-scale settings where exact training at every iteration is often computationally unnecessary. A natural extension arises when the learning problem is solved exactly, leading to joint hyperparameter and model-parameter optimization (HPO--MPO) formulations, which also fit the general framework considered here. In practice, however, these components are typically treated separately, with an emphasis on predictive accuracy alone. Jointly updating hyperparameters and model parameters in a coupled scheme remains a promising direction for future research.



\subsection{Contributions and Outline}
This study addresses the joint learning and optimization problem, which has primarily been examined within convex settings in prior research. The central contribution of our work lies in extending this analysis to cases where the outer optimization problems are pseudoconvex, an important step for applications such as decarbonization in energy systems and revenue maximization in retail portfolio management. Specifically, we propose a novel framework that alternates between optimization and learning updates, where the learning tasks need not be solved to full precision, thereby offering a more practical and computationally efficient solution. In addition to providing theoretical guarantees, we conduct extensive tests on real-world datasets, demonstrating the applicability of our proposed schemes in practical settings. We note that our theoretical results apply only to single-objective optimization. In the multiobjective setting, we adopt weighted scalarizations, which are known to be effective primarily when the underlying Pareto front is convex and may fail to recover solutions on nonconvex fronts.
The study also explores a broader range of problems in machine learning and renewable energy, underscoring the versatility of our approach. Furthermore, we examine the impact of prematurely terminating the learning process, elucidating how such an approach can still yield reasonable solutions to the outer optimization problems, thereby balancing computational cost with solution quality.

The paper is organized as follows. Section~\ref{sec:form} expands on the problem formulation, detailing the relevant approximations in the context of our study. Section~\ref{sec:algs} presents our alternating algorithm along with its convergence theory. Section~\ref{sec:numerics} provides a comprehensive numerical analysis, testing the algorithm rigorously on real-world applications. Section~\ref{sec:conclusion} concludes the study by summarizing the key findings and outlining potential avenues for future work.
\section{Problem Setup and Approximations}
\label{sec:form}
A central theme of this paper is the shared structural foundation underlying the two applications we study. Both can be expressed within a unified, two-level data-centric framework. At the inner level, we train machine learning models on historical data to obtain accurate and well-calibrated predictions. These predictive insights feed into the outer level, where the goal is to optimize high-level objectives—such as revenue or system performance—based on the learned models. A salient feature of our applications is that the resulting outer-level objectives often exhibit pseudoconvexity, which enables efficient optimization. In many instances, these objectives arise as fractional (ratio-based) functions, further shaping the mathematical structure of the problem. In this work, we undertake a systematic study of these problem classes, analyzing their formulations and implications within a joint predictive–prescriptive paradigm, where prediction is handled by machine learning and prescription by optimization. This interplay forms a foundational component of decision-making across industries, public-sector operations, and societal systems—though it often receives less explicit attention. We highlight this perspective through two representative applications in which this coupling is both natural and essential. We begin by summarizing fundamental principles of multiobjective optimization, which serve as a basis for the subsequent problem formulation. 

\subsection{Multi-objective optimization} \label{sec:moo_intro}
A multi-objective optimization problem (MOP) seeks to simultaneously optimize multiple, often conflicting, objective functions. Formally, it can be expressed as
\begin{align*} 
\min_{x \in \Omega} \quad F(x) = \left[ f_{1}(x), \hdots, f_{m}(x) \right],
\end{align*}
where $x \in \mathbb{R}^n$ is the decision variable, $\Omega \subset \mathbb{R}^n$ is the feasible set, and $F(x):\Omega \to \mathbb{R}^m$ denotes the objective vector containing $m \ge 2$ objective functions.

Since the objectives $\{f_i\}_{i=1}^m$ typically conflict with one another, improving one objective often requires sacrificing performance in another. Consequently, a single solution that simultaneously minimizes all objectives rarely exists. Instead, optimality is defined through the Pareto dominance.

\begin{definition}[Pareto dominance] \label{def:MOO_dominate}
A point $x_a \in \Omega$ is said to \textbf{dominate} another point $x_b \in \Omega$ (denoted $x_a \succcurlyeq x_b$) if and only if:
\begin{enumerate}
    \item $f_i(x_a) \le f_i(x_b)$ for all $i \in \{1, \ldots, m\}$; and
    \item $f_j(x_a) < f_j(x_b)$ for at least one $j \in \{1, \ldots, m\}$.
\end{enumerate}
In other words, $x_a$ is no worse than $x_b$ in every objective and strictly better in at least one.
\end{definition}

This notion leads to the following definitions.

\begin{definition}[Pareto optimality and Pareto front] \label{def:MOO_pareto_optimal}
A feasible point $x^* \in \Omega$ is \textbf{Pareto optimal} if no other feasible point dominates it. 
The \textbf{Pareto optimal set} consists of all such non-dominated points in $\Omega$. 
The \textbf{Pareto front} is the image of the Pareto optimal set under the mapping $F$, i.e., the set of corresponding objective vectors.
\end{definition}

The Pareto front captures the optimal trade-off set of solutions. The goal of a MOP algorithm is not to find a single optimal solution, but to generate a set of solutions that closely approximate this front. Ideally, these solutions should approximate the front as closely and uniformly as possible.

\subsubsection{Algorithms and metric} 

Solving a MOP involves generating or approximating Pareto optimal solutions. Broadly, methods fall into two categories: scalarization-based approaches and population-based metaheuristics.

\paragraph{Scalarization methods:} These methods convert the MOP into a single-objective problem that can be solved using standard optimization techniques.

\begin{itemize}
\item \textbf{Weighted sum method}~\cite{ehrgott2002multiobjective,kim2006adaptive}: Objectives are combined into a single scalar function:
\begin{align*}
    \min \hspace{1mm} F_{ws} (x) = \sum_{i=1}^m w_i f_i(x),
    \end{align*}
where $w_i \ge 0$ and $\sum_{i=1}^m w_i = 1$. Although simple and widely used, the weighted-sum method generates Pareto-optimal solutions only when the underlying Pareto front is convex; it fails to recover solutions in non-convex portions of the front~\cite{steuer1983interactive,dachert2012augmented}.
\item \textbf{$\epsilon$-constraint method}~\cite{mavrotas2009effective,deb2016multi}: In this method, one objective is minimized while the remaining objectives are enforced as constraints:
\begin{align*}
&\min \hspace{1mm} f_i(x) \\
\text{subject to} \quad & f_j(x) \le \epsilon_j, \quad \forall j \neq i.
\end{align*}
By varying the parameters $\epsilon_j$, one can generate different Pareto-optimal solutions, including those lying on nonconvex portions of the Pareto front. Although more computationally demanding, this approach is generally more robust than weighted scalarization methods for problems with nonconvex trade-off surfaces. 
\end{itemize}

\paragraph{Population-based metaheuristic methods:} Evolutionary algorithms (EAs) maintain a population of solutions and iteratively apply selection, crossover, and mutation to evolve this population toward the Pareto front. They can approximate the entire front in a single run. Prominent examples such as NSGA-II~\cite{deb2002fast, deb2011multi} employ non-dominated sorting and diversity-preserving mechanisms such as crowding distance.

\paragraph{Hypervolume:} The hypervolume (HV) indicator~\cite{zitzler1998multiobjective, zitzler2003performance} is a widely used performance metric. It measures the Lebesgue volume of the objective space dominated by a set of non-dominated solutions with respect to a reference point. Larger hypervolume values correspond to solution sets that are closer to the ideal vector and/or more diverse across the Pareto front. Importantly, HV is Pareto compliant and provides a strict ordering over solution sets. We remark that our focus in this work is on the weighted formulation. Our aim is not to undertake a comparative study of multiobjective optimization methods, but rather to adopt a theoretically sound and practically efficient approach for formulating and solving the joint predictive--prescriptive problem.
\subsection{Energy} \label{sec:intro_energy_problems}
Typical problems in energy management naturally adopt a multi-level structure. For clarity and completeness, we restate the formulation introduced earlier:
\begin{align*} 
\min_{x,\theta} \hspace{2mm} F(x,\theta) &= \left(f_{1}(x, \theta^{*}),..., f_{m}(x, \theta^{*}) \right), \quad
\text{where} \quad 
\theta^* = \argmin_{\theta \in \Theta} \hspace{1mm} h(\theta, y^{\text{fea}}, y^{\text{true}}).
\end{align*}
The minimization over $\theta$ can be further decomposed into two components: a hyperparameter optimization (HPO) stage concerned with selecting an appropriate learning architecture, and a model-parameter optimization (MPO) stage corresponding to training the model. Consequently, practical machine learning pipelines in energy systems often involve more than two coupled optimization levels. For discussions of frameworks that integrate both HPO and MPO, we refer the reader to~\cite{aswin2021, li2024algorithm}. In the present work, however, we restrict attention to the MPO component and do not explicitly address hyperparameter optimization in the energy setting.  
At the outer level, the decision variable $x$ represents quantities such as generation levels, unit commitment decisions, or sales, while the inner learning problem provides estimates of system parameters---including cost coefficients and capacity values---inferred from historical data and encoded in~$\theta$.

The Optimal Power Flow (OPF) problem has been studied extensively from multiple perspectives. A substantial portion of the literature focuses on the continuous variant, where start-up, shut-down, and minimum up/down-time constraints are ignored~\cite{kannanoms13,4form12,4form13,metzler03nash-cournot,kamat04twosettlement}. Across such formulations, the central operational requirement is the maintenance of supply--demand balance, while the typical objectives include minimizing generation costs, increasing renewable penetration, and addressing the potential for shortfall~\cite{kannanoms13}. Shortfall considerations usually introduce penalties when the generated power falls short of what is committed or when demand is not fully met. The constraints in these models often involve capacity limits, ramp-type inequalities, and the fundamental supply--demand balancing condition. Another line of work examines unit-commitment-based OPF, which captures generator switching behavior~\cite{updownoren10,birge04,ramp08,shahid95cons}. In thermal systems such as coal-based plants, generators cannot be turned off immediately after being switched on, and this temporal coupling requires the inclusion of additional integer variables. Given the resulting increase in complexity, the mixed-integer unit commitment formulation is not considered here and is instead reserved for future work. In this work, we restrict attention to the continuous OPF variant, which can be expressed as:
\begin{align*}
&\hspace{10mm} \min_{x} \hspace{1mm}  f_1(x, \theta^*) = \sum_{i \in \mathcal{N}_1} \left(a_{1} x_i + a_2 (x_i)^2 \right) \\
\text{subject to:} \\
& \sum_{i\in \mathcal{N}_2} x_i^t + \sum_{i\in \mathcal{N}_1} x_i^t(\theta^*) \geq d^t, \quad \forall t \in T, 
&& \text{supply--demand balance},\\[2mm]
& 0 \leq x_i^t \leq cap_i^t, \quad \forall i \in \mathcal{N}_2, \hspace{1mm} \forall t \in T,
&& \text{capacity constraints},\\[2mm]
& |x_i^t - x_i^{t-1}| \leq \delta x_{i}^{t-1},\quad \forall i \in \mathcal{N}_2, \hspace{1mm} \forall t \in T,
&& \text{ramp constraints},
\end{align*}
where $a_{1}$ and $a_{2}$ are positive coefficients, $\mathcal{N}_{1}$ denotes the set of renewable generators, $\mathcal{N}_{2}$ denotes the set of non-renewable generators, and $x_i(\theta)$ with $i\in \mathcal{N}_1$ represents the predicted renewable generation, parametrized by~$\theta$.
We consider a single agent that owns generation facilities and is allowed to sell power at all nodes, with generation at node $i$ denoted by $x_i$. For simplicity, we restrict our attention to a single objective of cost minimization. Although transmission constraints are not incorporated in the present formulation, the proposed algorithmic framework is intended to be system-level and general enough to accommodate such additions. In particular, transmission constraints based on DC power-flow models can be included without difficulty~\cite{metzler03nash-cournot,kannanoms13}.

\paragraph{Multi-objective OPF} 
Recent works have increasingly examined both OPF and unit-commitment settings through the lens of multiple objectives~\cite{alrasheedi2023investing, huang2022dual}. In addition to traditional economic considerations, objectives related to carbon emissions and renewable-energy penetration are now routinely incorporated. Emission-related objectives often resemble generation-cost models and typically take a quadratic form. Objectives that promote renewable-energy penetration naturally lead to fractional structures, which can be expressed as follows.
\begin{align*}
f_2(\theta) = \frac{\sum_{i \in \mathcal{N}_{1}} x_i(\theta)}{\sum_{i \in \mathcal{N}_{2}} x_i}. 
\end{align*}
We emphasize that both the non-renewable generation costs $f_1$ and the renewable-generation prediction model depend on the parameter~$\theta^*$, which is obtained by solving the associated machine learning problem. In the next section, we turn to the learning problem itself and discuss the estimation of the model parametrized by~$\theta$.



\paragraph{Power Prediction Problem:}
In our setting, the available data consist of pairs $(y_{\text{fea}}, y_{\text{true}})$, where $y_{\text{fea}}$ represents the features and $y_{\text{true}}$ denotes the corresponding ground-truth values. The goal is to learn a machine learning model $M(\theta)$ that best fits this data. Depending on the application, $M(\theta)$ may range from simple parametric forms such as linear or polynomial models to more complex architectures including decision trees, neural networks, and transformers. When the primary focus is predictive accuracy, the learning task reduces to solving
\begin{align*}
\min_{\theta} & \hspace{1mm} \sum_{i=1}^{N} d_i(M(\theta, y_i^{\text{fea}}), y_i^{\text{true}}) \\ 
\text{subject to:} & \hspace{1mm} \theta \in \theta(x).
\end{align*}
Here, $d_i(\cdot)$ denotes a distance or discrepancy measure between prediction and truth. A canonical choice is the squared-error loss:
\begin{align*}
    d_i \left( M \left(\theta, y_i^{\text{fea}} \right), y_i^{\text{true}} \right) = \left(M \left(\theta, y_i^{\text{fea}} \right) - y_i^{\text{true}} \right)^2.
\end{align*}
However, $d_i$ is not restricted to accuracy-based measures alone. Alternatives include metrics such as mean bias error or sparsity-promoting penalties, and these can be used either independently or in combination with accuracy terms. Accuracy itself may also be represented through cross-entropy losses, log-loss formulations, or $\ell_1$-type objectives. The reader is referred to~\cite{li2024algorithm, aswin2021} for further discussion of such loss structures. Multi-objective formulations, such as those considered in~\cite{aswin2021}, also provide a flexible route for balancing different performance criteria. In the present work, we restrict our attention to accuracy-based losses, although the overall framework readily accommodates general combinations of objectives.

Some illustrative examples of commonly used and custom-designed loss functions are shown below. Custom losses refer to user-defined objectives that go beyond the standard set typically supported by current machine learning libraries. For additional discussion on such practical aspects, the reader is referred to~\cite{aswin2021}. Note that $\|\theta\|^2$ in the expression below serves as a sparsity-promoting term.
\begin{align*}
& d(M(\theta,y^{fea}),y^{true}) = \|M(\theta,y^{fea})-y^{true}\|^2 + \|\theta \|^2,  \hspace{5mm} \text{Custom-Loss,}\\  
& d(M(\theta,y^{fea}),y^{true}) = -\sum_{i=1}^{N} y_i^{true} \log{(M(\theta,y_i^{fea}))}-\sum_{i=1}^{N} (1-y_i^{true}) \log{(1-M(\theta,z,y_i^{fea}))},  \hspace{5mm} \text{Log Loss.}
\end{align*}

\subsection{Retail Optimization}
Revenue maximization is central to any product portfolio, right from stock markets to retail. While there are multiple problems of interest from a portfolio perspective, we look at two key aspects, one arising in a consumer setup and the other stemming in the case of commodity markets. The case in the consumer setup may be referred to as Bertrand pricing, where product prices influence sales, which in turn contribute towards revenue. Commodity markets on the other hand operate by quantity of products sold. In these cases, prices are dictated by production quantities and sales. Examples are Nash-cournot based game theoretic electricity markets and financial stock markets. These operate on the underlying phenomenon of market clearing, where a price is arrived at based on bidding in quantities or stocks. Given that the scope of our work is not in game theory, we focus on the following optimization-related setups.   

\subsubsection{Retail Pricing}
Logit models~\cite{harker90competition,retail-best-14} have predominantly been the mainstay of literature. Similar to price $p$, we note that product attributes  $a_p$ like design specification become very relevant to the question. An example of a binary logit model is given as follows.
$$d(p,a_p) = \eta_p \left(1+e^{-U(p,a_p)}\right)^{-1}.$$
The utility function $U(.)$ is usually linear or quadratic. Consider a winter-wear product (shawl or gloves). Let $p$ and $a_1$ denote the price and number of woolen layers respectively. Then, the utility function may take the form $U(p,a_p) = 0.1p-0.01a_1$. The equation above is a case restricted to a single-firm with the customer presented with two options, namely purchase (high utility) and no-purchase (the scalar term ``1'' indicates the no-purchase option or no-utility). The purchase levels are fractions or probability levels, which are in turn multiplied by the market population $\eta_p$. We note that this extends to a multinomial-logit model~\cite{harker90competition} when there are multiple products, each with an exponential component defined in the denominator. However, we do not focus on product attributes in this work. For the current version of our work, we only focus on prices and not product attributes. 

We consider a firm with multiple products $N$,
indexed by $i$, price denoted by $p_{i}$ across a consumer segment. For completeness, we present both demand functions of the MNL model (i.e., there exists competition among products) and the binary logit model (i.e., no competition and only two options: buy or do not buy) with product $i$ as below:
\begin{align} \label{eq:demand_function}
d_i(p) = \frac{\exp \left(U_i\right)}{1+ \sum_{j=1}^{N} \exp \left( U_j \right)} \hspace{2mm} \text{and} \hspace{2mm} d_i(p) = (1 + \exp{(-U_i)})^{-1}
\end{align}
where $U_i=\theta_{i,0} p_i + \theta_{i,1}$, and $\theta_{i,0}$ and $\theta_{i,1}$ are parameters to be estimated (e.g., $\theta_{i,0}$ would be the single price sensitivity parameter).

Our formulation focuses on pricing decisions and incorporates temporary price reductions in the form of promotions, which are assumed to be known in advance. Following prior work~\cite{kannan2020computerizedsecond, cohen17, kannan2023solving}, the resulting revenue maximization problem is formulated as follows.
\begin{align*}
&\hspace{10mm} \min_{p} \hspace{1mm} f(p) = \sum_{i \in N} \sum_{t \in T} -p_i^t d(\theta_i^*, p_i^t) \\
 \text{subject to:} \hspace{2mm} 
& \rho_{i}^{low} \leq  p_i^t \leq  \rho_{i}^{upp}, \:\qquad \forall i \in N, \hspace{1mm} \forall t \in T, \hspace{3mm} & \text{Promotion/Markdown constraints.} \nonumber 
\end{align*}
Here, $p_i^t$ refers to the price of the product $i$ at time $t$, $p_i^{low}$ and $p_i^{upp}$ represent the lower and upper price bounds for product $i$, respectively. 
During several times of the year, stores do not receive any replenishments due to regulations, transportation issues, and sudden power outages. Discounts, i.e., promotions and markdowns, are mostly seasonal and are offered at specific sets of time $T$. An example of inter-item constraints is the following: If milk is priced at 2 Euros per liter, then yogurt should be priced at the least around 1.5 Euros per liter. While we assume that the demand is estimated from a logit model, other approximations also exist, depending on the problem context. The reader is asked to refer to~\cite{demand-survey13} for more details. 

\paragraph{Learning Problem:}
The logit model has been widely used in machine learning beyond demand estimation, with prominent examples including logistic regression and the cross-entropy loss. In classification settings, model parameters are typically estimated by maximizing the log-likelihood, or equivalently by minimizing the negative log-likelihood; however, this likelihood-based formulation is inherently tied to probabilistic classification and is not central to the computational optimization framework considered here. 

In the context of demand estimation, the logit model can instead be viewed as a smooth nonlinear mapping parameterized by $\theta$~\cite{mcfadden84, logit23, Kreimeier2023}. We estimate $\theta$ by minimizing a mean squared error objective, which yields a continuously differentiable loss with a well-defined gradient and Hessian structure, making it well suited for algorithmic analysis and large-scale numerical optimization.
\begin{align*}
\min_{\theta} \hspace{2mm} L(\theta) = \frac{1}{N\cdot T}\sum_{i=1}^{N}  \sum_{t=1}^{T} \left( (1+ \exp{(\theta^0_i p_{i}^t + \theta^1_i}))^{-1} -  y_{i,t}^{true} \right)^2.
\end{align*}
Here, $T$ denotes the number of observations per product, and $y^{true}_{i,t}$ refers to the observed sales of product $i$ at time $t$. While conceptually straightforward, this formulation is more challenging to handle computationally. We therefore propose a modified variant obtained by applying a logarithmic transformation, which results in a more tractable and convex objective:
\begin{align*}
\min_{\theta \in \Theta} \hspace{1mm} h(\theta) = \frac{1}{N\cdot T} \sum_{i=1}^{N}  \sum_{t=1}^{T} \left( \theta_{i,0} p_{i}^t + \theta_{i,1} -\log{\left(\frac{1}{y_{i,t}^{true}}-1\right)}\right)^2 = \frac{1}{N\cdot T} \sum_{i=1}^{N} \|A_i \theta - b_i \|_2^2.
\end{align*}
Here, $A_i \in \mathbb{R}^{T \times 2}$ and $b_i \in \mathbb{R}^{T}$ are constant matrices and vectors defined appropriately. The resulting problem is a linear regression model, which can be solved either in closed form or iteratively; in this work, we adopt an iterative solution approach.

It is straightforward to observe that the solution to the above problem can be obtained in a single step by solving either a linear system or a sequence of linear systems (equivalently, a sequence of equality-constrained quadratic optimization problems). However, this corresponds to a textbook, offline setting. In many practical applications, the problem is inherently online, in the sense that new samples arrive continuously and the dataset grows over time. 
That implies recomputing $\theta^*$ from scratch after each batch processing would require repeatedly constructing and solving increasingly large linear systems, which would consume extensive time and memory.  

A representative example is recommender and assortment systems in online retailing/advertising, where user feedback (clicks, purchases) is observed, and the platform needs to simultaneously (a) learn choice-model parameters and (b) optimize decisions based on the current estimates. In such applications, the binary logit model is a natural choice: the purchase probability is modeled as a logistic function of item features, and the model has to be updated frequently to track evolving user preferences. In systems operating at the scale of Amazon or Netflix, the dataset is massive, and full retraining ``in one shot'' after each data update is typically infeasible. Furthermore, this approach has another drawback that the final model may already be outdated by the time it is deployed~\cite{bottou2018optimization}. This is a common motivation for incremental/online learning methods, which update model parameters in deployed systems using a small number of recent samples and cheap gradient-type steps rather than full retraining~\cite{mcmahan2013ad, losing2018incremental}.

Moreover, the learning problem is coupled to decision-making, which further motivates our joint learning-optimization method as in~\eqref{alg:jointalg}. Specifically, market conditions are frequently driven by dynamic competitor actions~\cite{harker90competition}. In online sales, competitor pricing fluctuates rapidly. Waiting to retrain models from scratch would cause inventory and revenue losses. Consequently, modern frameworks tend to use online learning methods, where parameters are updated through stochastic gradients derived from recent samples. This allows the model to adapt continuously to changing demand signals~\cite{rusmevichientong2010dynamic}.

We note that in our case $\Theta$ is a very simple set with bound constraints defined as follows. 
\begin{align*}
    \Theta = \left\{ \left\{\theta_{i,0}, \theta_{i,1} \right\}_{i=1}^N | \theta_{i}^0 \geq 0, \forall i \in \{1,...,N\} \right\}.
\end{align*}

\section{Algorithms}
\label{sec:algs}
As discussed earlier, the two-level problem structure can be approached in multiple ways. A natural starting point is to solve the learning problem to completion and then solve the resulting optimization problem. An alternative is to alternate between the learning and optimization components by solving the learning subproblem to a prescribed accuracy, updating the optimization variables, and repeating this process. For pseudoconvex or fractional-type problems, straightforward gradient-based schemes do not generally guarantee convergence, which motivates the use of extragradient-type methods. While such schemes have been studied in deterministic~\cite{Korpelevich_1976} and stochastic settings~\cite{kannan2019optimal}, their applicability to coupled \emph{joint optimization and learning} remains an open question. Motivated by applications such as portfolio optimization and unit commitment—where fractional or pseudoconvex structures naturally arise and classical gradient updates may be inadequate—we evaluate the following algorithmic strategies:
\begin{itemize}
\item Solve the learning problem fully, and then solve the optimization problem.
\item Solve the learning problem to a fixed accuracy, solve/update the optimization problem, and iterate between the two.
\item Propose a coupled scheme consisting of a single extragradient step at the outer level and a single gradient step at the inner level.
\end{itemize}
In what follows, we concentrate on the third approach and develop the proposed coupled scheme. Before presenting the algorithm, we briefly review the required preliminaries.
\paragraph{Basics:}
Given a probability space $(\Omega, \mathcal{F}, P)$, where $\Omega$ denotes the sample space, $\mathcal{F}$ is a $\sigma$-algebra on $\Omega$, and $P$ is a probability measure on $(\Omega,\mathcal{F})$, we consider a mapping
\[
    \chi : \Omega \to \mathbb{R}^n.
\]
We assume throughout that $\chi$ is a random variable, and all expectations are taken with respect to its induced distribution. This assumption is standard and follows directly from classical results in measure theory; see, e.g., any standard reference on probability spaces and measurable mappings. We define the learning problem as
\begin{align*}
    \min_{\theta \in \Theta} \; \mathbb{E}\!\left[ g(\theta,\chi) \right],
\end{align*}
where $\Theta \subseteq \mathbb{R}^n$ is a convex set and the function $g(\cdot,\chi)$ is strongly convex for every realization of $\chi$.  
For notational convenience, we define
\[
    h(\theta) \triangleq \mathbb{E}\!\left[ g(\theta,\chi) \right], 
    \quad \forall \theta \in \Theta.
\]
The stochasticity of the learning problem arises entirely through the random variable $\chi$, while the optimization variable $\theta$ is deterministic. This viewpoint will be adopted throughout the paper.

\subsection{Algorithmic Framework}\label{sec:single_step_alg}
We first consider the case in which a single update is performed sequentially in the learning and optimization spaces. This single-step framework serves as a baseline and is later extended to a multi-step variant. In the present setting, the outer-level optimization variable $x$ is treated deterministically, while stochasticity is incorporated only in the learning variable $\theta$ to capture uncertainty arising from data-driven model estimation. This modeling choice is motivated by the applications considered. In retail pricing, effects such as demand variability and seasonality are implicitly captured through the learned demand model, while the resulting pricing optimization is deterministic given the current estimates. Similarly, in optimal power flow, uncertainty primarily arises from renewable generation (e.g., wind), whereas conventional generation costs and demand profiles exhibit predictable patterns and are handled through learning-based prediction. As a result, the outer-level problem can be reasonably approximated as deterministic in most practical settings. Extensions that explicitly incorporate stochasticity in the $x$-updates are possible and left for future work.

Our basic joint learning–optimization algorithm takes the following form. For a prelude on related joint algorithms in purely convex settings, we refer the reader to~\cite{kannan2015distributed}.
\begin{align}\label{alg:jointalg}
x_{k+1/2} & = \Pi_{X} \left(x_k-\gamma_k \nabla f(x_k,\theta_k) \right),    \\ 
x_{k+1} & = \Pi_{X} \left(x_k-\gamma_k \nabla f(x_{k+1/2},\theta_k) \right), \nonumber \\ 
\theta_{k+1} & = \Pi_{\Theta} \left( \theta_k -\beta_k \left(\nabla h(\theta_k)+w_k\right)\right). \nonumber
\end{align}
Ideally, the iterate $\theta_k$ serves as a surrogate for the true optimizer $\theta^*$ of the learning problem. Intuitively, $\theta_k$ evolves on a slower timescale and gradually approaches $\theta^*$ as the algorithm progresses. Prior to developing the convergence theory for the above scheme, we state the assumptions imposed on the steplength sequences and the objective functions.

\begin{assumption} Steplength Sequences. 
\begin{itemize}
\item[(a)] The steplength sequences for updating the decision variables are square summable and not summable. 
\begin{align*}
   \sum_{k=0}^{\infty} \gamma_k = \infty, \hspace{2mm} \sum_{k=0}^{\infty} \gamma_k^2 < \infty. 
\end{align*}
\item[(b)] The steplength sequences for updating the learning parameters are square summable and not summable.
\begin{align*}
    \sum_{k=0}^{\infty} \beta_k = \infty, \hspace{2mm} \sum_{k=0}^{\infty} \beta_k^2 < \infty,
\end{align*}
which implies $\lim_{k\to \infty} \beta_k=0$. Moreover, we set $\beta_k \leq \beta_0 \leq \frac{2\mu_h}{L_h^2}$ for all $k$. 
\item[(c)] For some $\tau \in (0,1)$, it holds that
\begin{align*}
    \sum_{k=0}^{\infty} \gamma_k^{2-\tau} < \infty, \quad 
    \lim_{k\to \infty} \frac{\gamma_k^{\tau}}{\beta_k} = 0.
\end{align*}
\end{itemize}
\label{assumption:A1steps}
\end{assumption}

\begin{assumption} 
Gradient Specifications.
\label{assumption:objs}
\begin{itemize}
\item[(a)] The objective $f$ is pseudoconvex, i.e., for any $x_1,x_2 \in X$, 
\begin{align*}
    \nabla f(x_1)^T (x_2-x_1)\geq 0 \Longrightarrow f(x_2)\geq f(x_1).
\end{align*}
\item[(b)] The function $h$ is strongly convex with parameter $\mu_h>0$, i.e., 
\begin{align*}
    \left( \nabla h(\theta_1)-\nabla h(\theta_2)\right)^T \left( \theta_1-\theta_2 \right) \geq \mu_h||\theta_1 - \theta_2||^2
\end{align*}
for any $\theta_1, \theta_2\in \Theta$.
\item[(c)] The gradients of the objectives  $\nabla f$ and $\nabla h$ are Lipschitz continuous with positive constants $L_f>0$ and $L_h>0$ respectively. In other words, for any $x_1,x_2 \in X$ and $\theta_1,\theta_2 \in \Theta$, it holds
\begin{align*}
    ||\nabla f(x_1, \theta)-\nabla f(x_2, \theta)|| &\leq L_f ||x_1-x_2||, \\
    ||\nabla f(x, \theta_1)-\nabla f(x, \theta_2)|| &\leq L_{\theta} ||\theta_1-\theta_2||, \\
    ||\nabla h(\theta_1)-\nabla h(\theta_2)|| &\leq L_h ||\theta_1-\theta_2||.
\end{align*}
Moreover, $||\nabla f(x, \theta^*)||\leq B$ for all $x\in X$, where $B>0$ is a scalar and $\theta^* = \argmin_{\theta \in \Theta} \hspace{1mm} h(\theta)$.
\end{itemize}
\end{assumption}
We next turn to the convergence analysis of the proposed algorithm. Before presenting the formal results, we introduce a standard assumption on the stochastic error arising in the learning updates. This assumption imposes boundedness on the first and second conditional moments of the noise sequence and ensures well-posedness of the stochastic approximation scheme. Similar assumptions have been extensively studied in the literature; see, for instance,~\cite{kannan2019optimal,kannan2015distributed}. We restate the assumption here for completeness.
\begin{assumption} \label{assumption:first_second_moment} Boundedness on Uncertainty: For all $k$, the conditional first moments $\mathbb{E}[w_k | \mathcal{F}_k]=0$ and the conditional second moments $\mathbb{E}[||w_k||^2 | \mathcal{F}_k]\leq \nu^2$, where constant $\nu >0$.
\end{assumption}
We now recall two classical convergence results from stochastic approximation theory that will be used repeatedly in the subsequent analysis.
\begin{lemma} (Lemma $10$ in \cite{polyak1987introduction})\label{lemma:convergence_as_1}
    Suppose $V_k$ is a sequence of nonnegative random variables to $\mathcal{F}_k$. If  $\mathbb{E}\left[V_{k+1}|\mathcal{F}_k \right] \leq (1-\delta_k) V_k +\xi_k$, where scalars $0 \leq \delta_k\leq 1, \xi_k>0$ satisfy $\sum_{k=0}^{\infty} \delta_k=\infty$, $\sum_{k=0}^{\infty} \xi_k<\infty$ and $\lim_{k\to \infty}\frac{\xi_k}{\delta_k}=0$. Then $V_k \to 0$ almost surely and $\mathbb{E}\left[ V_k \right]\to 0$ as $k \to \infty$.
\end{lemma}

\begin{lemma} (Robbins--Siegmund, Lemma $11$ in \cite{polyak1987introduction}) \label{lemma:convergence_as_2}
    Suppose $V_k, u_k, \delta_k, \xi_k$ are nonnegative random variables. If  $\mathbb{E}\left[V_{k+1}|\mathcal{F}_k \right]\leq (1+\delta_k) V_k - u_k +\xi_k$ almost surely, where $\sum_{k=0}^{\infty} \delta_k<\infty$ almost surely and $\sum_{k=0}^{\infty} \xi_k<\infty$ almost surely. Then $\sum_{k=0}^{\infty} u_k < \infty$ almost surely and $ V_k \to V^*$ as $k \to \infty$, where $V^* \geq 0$ is some random variable.
\end{lemma}

\paragraph{Intuition:}The proposed algorithm is motivated by the need to balance learning accuracy with computational efficiency while retaining theoretical robustness in pseudoconvex regimes. Rather than insisting on full convergence of the learning problem at every stage, we allow the learning variable to evolve gradually alongside the optimization variable. The optimization step is therefore guided by a \emph{moving surrogate} of the true model, reflecting the current state of learning. To address the pseudoconvex or fractional structure of the outer problem—where naive gradient descent may fail—we incorporate an extragradient correction that anticipates curvature and stabilizes the update. The resulting scheme couples a cautious, predictive update in the decision space with a standard stochastic gradient step in the learning space, operating on different timescales. This design enables consistent progress toward optimality without requiring exact solutions of the learning problem at intermediate stages, thereby providing a practical and theoretically sound framework for joint learning and optimization.

\subsubsection{Convergence Analysis} The main embodiment of our analysis lies in establishing relationships between successive iterates of the algorithm. To this end, we couple the sequences $\{x_k\}$ and $\{\theta_k\}$ into a single composite quantity and derive inequalities that link the iterates across consecutive steps. These relations capture the interaction between the optimization and learning components of the algorithm and reflect the impact of inexact learning on the optimization updates. By carefully bounding the resulting recursion and invoking Lemmas~\ref{lemma:convergence_as_1} and~\ref{lemma:convergence_as_2}, we are able to formally establish convergence of the proposed scheme. We begin this analysis with the following lemma, which relates successive iterates generated by the joint extragradient method.
 
\begin{lemma} \label{lemm:xiter}
Consider the iterates generated by~(\ref{alg:jointalg}). Let assumptions~\ref{assumption:A1steps} and~\ref{assumption:objs} hold. Then, the following hold for any iteration $k$. 
\begin{align*} 
||x_{k+1}-x^* ||^2 \leq (1+ t_{A,k}) ||x_k-x^* ||^2 - t_{B} ||x_k-x_{k+1} ||^2 + 2u_k +  t_{C} \gamma_k^2 + t_{D,k} ||\theta_k-\theta^* ||^2, 
\end{align*} 
where $u_k \leq0$, $t_{A,k}>0$, $t_{B}>0$, $t_{C}>0$, and $t_{D,k}>0$ are appropriately defined bounded scalars.
\end{lemma}

\begin{proof}
    We start with defining $y_k = x_k-\gamma_k \nabla f(x_{k+1/2}, \theta_k)$. Then, it can be observed that
    \begin{align} \label{eq:x_k_transform}
    ||x_{k+1}-x^* ||^2  &= ||\Pi_{X}(y_k)-y_k + y_k-x^*||^2 \nonumber \\
     &= ||\Pi_{X}(y_k)-y_k ||^2 + ||y_k-x^* ||^2 + 2(\Pi_{X}(y_k)-y_k)^T \left(y_k-x^*\right) 
     \end{align}   
     
     Based on Lemma $4$ in~\cite{kannan2019optimal}, it holds that
     \begin{align*}
     &\quad 2||\Pi_{X}(y_k)-y_k ||^2 + 2(\Pi_{X}(y_k)-y_k)^T \left(y_k-x^*\right) \\
     &=2||\Pi_{X}(y_k)-y_k ||^2 + 2(\Pi_{X}(y_k)-y_k)^T \left(y_k - \Pi_{X}(y_k)+ \Pi_{X}(y_k) - x^*\right) \\
     &=2||\Pi_{X}(y_k)-y_k ||^2 - 2||\Pi_{X}(y_k)-y_k ||^2 + 2(\Pi_{X}(y_k)-y_k)^T \left(\Pi_{X}(y_k) - x^*\right) \\
     &=2(\Pi_{X}(y_k)-y_k)^T \left(\Pi_{X}(y_k) - x^*\right) \leq 0,\\
     \text{which is equivalent to } &||\Pi_{X}(y_k)-y_k ||^2 + 2(\Pi_{X}(y_k)-y_k)^T \left(y_k-x^*\right) \leq -||\Pi_{X}(y_k)-y_k ||^2.
     \end{align*}

    Hence, (\ref{eq:x_k_transform}) becomes
     \begin{align}
     ||x_{k+1}-x^* ||^2 &\leq ||y_k-x^* ||^2-||y_k-\Pi_{X}(y_k) ||^2 \nonumber\\
     & = ||x_k-\gamma_k \nabla f(x_{k+1/2},\theta_k)-x^* ||^2 - ||x_k-\gamma_k \nabla f(x_{k+1/2},\theta_k) -x_{k+1}||^2 \nonumber \\
     & = ||x_k-x^*||^2 + \gamma_k^2 ||\nabla f(x_{k+1/2},\theta_k) ||^2 -2 \gamma_k (x_k-x^*)^T \nabla f(x_{k+1/2},\theta_k) \nonumber \\
     & \hspace{3mm} -||x_k- x_{k+1} ||^2 - \gamma_k^2 ||\nabla f(x_{k+1/2},\theta_k) ||^2 + 2\gamma_k (x_k-x_{k+1})^T \nabla f(x_{k+1/2},\theta_k) \nonumber \\
    &  = ||x_k-x^* ||^2 -||x_k-x_{k+1} ||^2 + 2\gamma_k (x^*-x_{k+1})^T \nabla f(x_{k+1/2},\theta_k). \label{ieq:x_k_x^*_convergent_1}
    \end{align}
For the last term, we introduce and subtract $\gamma_k x_k^{T}\nabla f(x_{k+1/2}, \theta_k)$, which yields
    \begin{align*}
    &\gamma_k (x^*-x_{k+1})^T \nabla f(x_{k+1/2},\theta_k) \\
    & \hspace{5mm} = \gamma_k (x^*-x_k)^T \nabla f(x_{k+1/2},\theta_k) + \gamma_k (x_k-x_{k+1})^T \nabla f(x_{k+1/2},\theta_k) \\
    & \hspace{5mm} = \gamma_k (x^*-x_k)^T \nabla f(x_{k+1/2},\theta^*)+\gamma_k (x^*-x_k)^T \left[ \nabla f(x_{k+1/2},\theta_k)-\nabla f(x_{k+1/2},\theta^*)\right] \\ 
    & \hspace{5mm} \hspace{2mm} + \gamma_k (x_k-x_{k+1})^T \nabla f(x_{k+1/2},\theta_k) \\ 
    & \hspace{5mm} = \underbrace{\gamma_k (x^*-x_k)^T \nabla f(x_k,\theta^*)}_{u_k} +  \gamma_k (x^*-x_k)^T \left[ \nabla f(x_{k+1/2},\theta^*) - \nabla f(x_k,\theta^*)\right] \\ 
    & \hspace{5mm} \hspace{2mm} + \gamma_k (x^*-x_k)^T \left[ \nabla f(x_{k+1/2},\theta_k) - \nabla f(x_{k+1/2},\theta^*)\right] + \gamma_k (x_k-x_{k+1})^T \nabla f(x_{k+1/2},\theta_k). 
    \end{align*}
    Note that $x^* = \argmin_x \hspace{1mm} f(x,\theta^*)$, which implies $u_k \leq 0$ for all $k$.
    For all constants $\eta_1, \eta_2, \eta_3, \tau >0$, (\ref{ieq:x_k_x^*_convergent_1}) further gives that
    \begin{align*}
    ||x_{k+1}-x^* ||^2 & \leq  ||x_k-x^* ||^2-||x_k-x_{k+1} ||^2 + 2 u_k + 2\gamma_k (x^*-x_k)^T \left[ \nabla f(x_{k+1/2},\theta^*) - \nabla f(x_k,\theta^*)\right] \\
    &\hspace{3mm} + 2\gamma_k (x^*-x_k)^T \left[ \nabla f(x_{k+1/2},\theta_k) - \nabla f(x_{k+1/2},\theta^*)\right] + 2\gamma_k (x_k-x_{k+1})^T \nabla f(x_{k+1/2},\theta_k) \\
    & \leq ||x_k-x^* ||^2-||x_k-x_{k+1} ||^2 + 2 u_k \\
    & \hspace{3mm} + \eta_1 \gamma_k^2 ||x_k-x^* ||^2 + \frac{1}{\eta_1} ||\nabla f(x_{k+1/2},\theta^*) -\nabla f(x_k,\theta^*)||^2  \hspace{10mm} \text{(by Young's inequality)} \\
    &\hspace{3mm} + \eta_2 \gamma_k^{2-\tau} ||x_k-x^* ||^2 + \frac{\gamma_k^{\tau}}{\eta_2} ||\nabla f(x_{k+1/2}, \theta_k) -\nabla f(x_{k+1/2},\theta^*)||^2 \\ 
    & \hspace{3mm} + \eta_3 ||x_k-x_{k+1} ||^2 + \frac{\gamma_k^2}{\eta_3} ||\nabla f(x_{k+1/2},\theta_k) ||^2 \\
    & \leq \left(1 + \eta_1 \gamma_k^2 + \eta_2 \gamma_k^{2-\tau} \right)||x_k-x^* ||^2 - ( 1-\eta_3)||x_k-x_{k+1} ||^2 + 2 u_k \\ 
    & \hspace{3mm} + \frac{L_f^2}{\eta_1} ||x_{k+1/2}-x_k ||^2 + \frac{L_{\theta}^2 \gamma_k^{\tau}}{\eta_2} ||\theta_k-\theta^* ||^2 + \frac{\gamma_k^2}{\eta_3}||\nabla f(x_{k+1/2},\theta_k) ||^2.
    \end{align*}

    Next, we examine term-by-term as follows. 
    \begin{align*}
    ||x_{k+1/2}-x_k ||^2 & = ||\Pi_{X}(x_{k}-\gamma_k \nabla f(x_k, \theta_k)) -\Pi_{X}(x_k)||^2 \\ 
    & \leq \gamma_k^2 || \nabla f(x_k,\theta_k) ||^2 \\
    & = \gamma_k^2 ||\nabla f(x_k,\theta^*) + \nabla f(x_k,\theta_k)-\nabla f(x_k,\theta^*) ||^2 \\ 
    &= \gamma_k^2 \left(||\nabla f(x_k,\theta^*) ||^2 + ||\nabla f(x_k,\theta_k)-\nabla f(x_k,\theta^*) ||^2 + 2 \nabla f(x_k,\theta^*)^T \left(\nabla f(x_k,\theta_k)-\nabla f(x_k,\theta^*) \right) \right) \\
    &\leq \gamma_k^2 \left(2||\nabla f(x_k,\theta^*) ||^2 + 2||\nabla f(x_k,\theta_k)-\nabla f(x_k,\theta^*) ||^2 \right)\\
    & \leq \gamma_k^2 \left( 2 B^2 + 2 L_{\theta}^2 ||\theta_k-\theta^* ||^2 \right).
    \end{align*}
    Similarly, 
    \begin{align*}
    || \nabla f(x_{k+1/2},\theta_k) ||^2 &\leq 2|| \nabla f(x_{k+1/2}, \theta^*)||^2 + 2|| \nabla f(x_{k+1/2},\theta_k) - \nabla f(x_{k+1/2}, \theta^*)||^2 \\
    &\leq  2 B^2 + 2L_{\theta}^2 ||\theta_k-\theta^* ||^2.
    \end{align*}
    This further leads to the following expression. 
    \begin{align} \label{eq:xk_bounded}
    ||x_{k+1}-x^* ||^2 & \leq \left( 1+ \eta_1 \gamma_k^2 + \eta_2 \gamma_k^{2-\tau} \right) ||x_k-x^* ||^2 - \left( 1-\eta_3\right) ||x_{k}-x_{k+1} ||^2 + 2u_k  \nonumber \\ 
    &\hspace{3mm} + \frac{2L_{f}^2 \gamma_k^2}{\eta_1} \left(B^2 + L_{\theta}^2 ||\theta_k -\theta^*||^2 \right) + \frac{L_{\theta}^2 \gamma_k^{\tau}}{\eta_2} ||\theta_k-\theta^* ||^2 + \frac{2\gamma_k^2}{\eta_3} \left(B^2 + L_{\theta}^2 ||\theta_k-\theta^* ||^2 \right) \nonumber \\
    & = \left( 1+ \underbrace{\eta_1 \gamma_k^2 + \eta_2 \gamma_k^{2-\tau} }_{t_{A,k}}\right) ||x_k-x^* ||^2 - \underbrace{\left( 1-\eta_3\right)}_{t_B} ||x_{k}-x_{k+1} ||^2 +2u_k \nonumber\\ 
    &\hspace{3mm} + \gamma_k^2 \underbrace{ \left( \frac{2L_f^2 B^2}{\eta_1} + \frac{2B^2}{\eta_3} \right)}_{t_{C}} + \underbrace{\left( \frac{2L_f^2 \gamma_k^2 L_{\theta}^2}{\eta_1} +\frac{L_{\theta}^2 \gamma_k^{\tau}}{\eta_2} + \frac{2\gamma_k^2L_{\theta}^2}{\eta_3} \right)}_{t_{D,k}} ||\theta_k-\theta^* ||^2. 
    \end{align}
Here, $\eta_3>0$ may be chosen accordingly to ensure that $t_B>0$ while simultaneously tightening the overall bound.
\end{proof}


\begin{lemma} \label{lemma:convergence_xk_theta_k}
Consider the iterates generated by~(\ref{alg:jointalg}). Let assumptions~\ref{assumption:A1steps} and~\ref{assumption:objs} hold. Then, the following hold for any iteration $k$. 
\begin{align*}
||x_{k+1}-x^* ||^2+||\theta_{k+1}-\theta^* ||^2 
&\leq (1+t_{\text{max}}) \left( ||x_k-x^* ||^2 + ||\theta_k-\theta^* ||^2 \right)  \\
&\hspace{3mm} - t_{B} ||x_k-x_{k+1} ||^2 + 2u_k + t_C \gamma_k^2 + \beta_k^2 ||w_k||^2  \\
&\hspace{3mm} + 2\beta_k^2 (w_k)^T \left( \nabla h(\theta_k)-\nabla h(\theta^*) \right) - 2\beta_k (w_k)^T \left( \theta_k - \theta^* \right), 
\end{align*}
where $t_{\text{max}} = \max\{t_{A,k},t_{E,k}\}$, $t_{B}>0$, $t_{C}>0$, and $u_k \leq 0$ are appropriately defined finite scalars.
\end{lemma}
\begin{proof}
    Note that $\theta^* = \prod_{\Theta} (\theta^* - \beta_k \nabla h(\theta^*))$. Analyzing the consecutive iterates in the space of $\theta$ by~(\ref{alg:jointalg}), we have the following
\begin{align*}
||\theta_{k+1}- \theta^*||^2 & \leq ||\theta_k - \theta^* - \beta_k \left( \nabla h(\theta_k)-\nabla h(\theta^*)\right) - \beta_k w_k ||^2 \\ 
& = ||\theta_k-\theta^* ||^2 + \beta_k^2 || \nabla h(\theta_k)-\nabla h(\theta^*) ||^2 +\beta_k^2 ||w_k||^2 -2\beta_k \left( \nabla h(\theta_k) - \nabla h(\theta^*) \right)^T \left(\theta_k - \theta^* \right)\\ 
&\hspace{3mm} + 2\beta_k^2 (w_k)^T \left( \nabla h(\theta_k)-\nabla h(\theta^*) \right) - 2\beta_k (w_k)^T \left( \theta_k - \theta^* \right).
\end{align*}
Based on Assumption~\ref{assumption:objs}, it is easy to obtain
\begin{align*}
||\theta_{k+1}- \theta^*||^2 & \leq ||\theta_k-\theta^* ||^2 + \beta_k^2 L_h^2 ||\theta_k-\theta^* ||^2 +\beta_k^2 ||w_k||^2 - 2\beta_k \mu_h ||\theta_k-\theta^* ||^2  \\ 
&\hspace{3mm} + 2\beta_k^2 (w_k)^T \left( \nabla h(\theta_k)-\nabla h(\theta^*) \right) - 2\beta_k (w_k)^T \left( \theta_k - \theta^* \right)\\
& = \left(1-2\beta_k\mu_h + \beta_k^2 L_h^2 \right) ||\theta_k-\theta^* ||^2 +\beta_k^2 ||w_k||^2  \\
&\hspace{3mm} + 2\beta_k^2 (w_k)^T \left( \nabla h(\theta_k)-\nabla h(\theta^*) \right) - 2\beta_k (w_k)^T \left( \theta_k - \theta^* \right).
\end{align*}

Combining the above with (\ref{eq:xk_bounded}), we have the following. 
\begin{align} \label{eq:xk_theta_k_bounded_1}
||x_{k+1}-x^* ||^2 + ||\theta_{k+1}-\theta^* ||^2 & \leq \left( 1+ t_{A,k} \right) ||x_k-x^* ||^2 - t_B ||x_k-x_{k+1} ||^2 +2 u_k \nonumber\\ 
&\hspace{3mm}  + t_C \gamma_k^2 + \left( t_{D,k} + 1 -2\beta_k \mu_h + \beta_k^2 L_h^2 \right) ||\theta_k-\theta^* ||^2 +\beta_k^2 ||w_k||^2 \nonumber \\ 
&\hspace{3mm} + 2\beta_k^2 (w_k)^T \left( \nabla h(\theta_k)-\nabla h(\theta^*) \right) - 2\beta_k (w_k)^T \left( \theta_k - \theta^* \right).
\end{align}
Note that
\begin{align*}
t_{D,k} + 1 - 2\beta_k \mu_h + \beta_k^2 L_h^2  &= \frac{2 L_f^2 \gamma_k^2 L_{\theta}^2}{\eta_1} +\frac{L_{\theta}^2 \gamma_k^{\tau}}{\eta_2} + \frac{2\gamma_k^2L_{\theta}^2}{\eta_3} + 1 - 2\beta_k \mu_h + \beta_k^2 L_h^2 \\
&= 1 + \underbrace{\gamma_k^2 \left( \frac{2L_f^2 L_{\theta}^2}{\eta_1} + \frac{2L_{\theta}^2}{\eta_3} \right)}_{t_{E,k}} + \underbrace{\frac{L_{\theta}^2 \gamma_k^{\tau}}{\eta_2} - 2\beta_k \mu_h +\beta_k^2 L_h^2}_{t_{F,k}}.
\end{align*}
On re-arrangement, (\ref{eq:xk_theta_k_bounded_1}) is equivalent to the following. 
\begin{align} \label{eq:xk_theta_k_bounded_2}
||x_{k+1}-x^* ||^2 + ||\theta_{k+1}-\theta^* ||^2 & \leq  \left( 1+ t_{A,k} \right) ||x_k-x^* ||^2 - t_B ||x_k-x_{k+1} ||^2 + 2 u_k \nonumber\\ 
&\hspace{3mm} + t_C \gamma_k^2 + \left(1+ t_{E,k} +t_{F,k}\right) ||\theta_k-\theta^* ||^2 +\beta_k^2 ||w_k||^2  \nonumber\\
&\hspace{3mm} + 2\beta_k^2 (w_k)^T \left( \nabla h(\theta_k)-\nabla h(\theta^*) \right) - 2\beta_k (w_k)^T \left( \theta_k - \theta^* \right) \nonumber\\
&\leq \left( 1+ \max\{t_{A,k}, t_{E,k}\} \right) \left( ||x_{k}-x^* ||^2 + ||\theta_{k}-\theta^* ||^2 \right) - t_B ||x_k-x_{k+1} ||^2 \nonumber\\
&\hspace{3mm} + 2 u_k + t_C \gamma_k^2 + t_{F,k} ||\theta_k-\theta^* ||^2 +\beta_k^2 ||w_k||^2 \nonumber\\
&\hspace{3mm} + 2\beta_k^2 (w_k)^T \left( \nabla h(\theta_k)-\nabla h(\theta^*) \right) - 2\beta_k (w_k)^T \left( \theta_k - \theta^* \right).
\end{align}

For the term $t_{F,k}$, we know $\lim_{k\to \infty} \beta_k=0$ and $\lim_{k\to \infty} \frac{\gamma_k^{\tau}}{\beta_k} = 0$ by Assumption~\ref{assumption:A1steps}. Then, 
\begin{align*}
    \lim_{k\to \infty} \left(\frac{t_{F,k}}{\beta_k} \right)
        = \frac{L_{\theta}^2}{\eta_2} \lim_{k\to \infty} \frac{ \gamma_k^{\tau}}{\beta_k} - 2 \mu_h + L_h^2 \lim_{k\to \infty} \beta_k =-2\mu_h < 0.
\end{align*}
Hence, there exists $K$ such that for all $k \geq K$, one has $\frac{t_{F,k}}{\beta_k} \leq -\mu_h$. Since $\beta_k >0$ for all $k$, it has
\begin{align*}
    t_{F,k} \leq -\mu_h \beta_k <0.
\end{align*}
In other words, $t_{F,k} < 0$ for all $k \geq K$. 
\end{proof}
\begin{proposition} \label{prop:xk_theta_k_convergence_as}
    Given $x_k,\theta_k$ and $\mathcal{F}_k$. Let sequences $\{x_k\}$ and $\{\theta_k\}$ be generated by (\ref{alg:jointalg}). Then the sequence $\{x_{k+1}\}$ converges almost surely to a point $x^*$ that satisfies the first-order (KKT) stationarity conditions for the problem $\min_{x \in X} f(x, \theta^*)$
\end{proposition}
\begin{proof}

    Taking conditional expectation with respect to $\mathcal{F}_k$, (\ref{eq:xk_theta_k_bounded_2}) follows
        \begin{align} \label{ieq:condi_expectation_xk_theta_k}
        \mathbb{E}\left[ ||x_{k+1}-x^* ||^2 + ||\theta_{k+1}-\theta^* ||^2 | \mathcal{F}_k \right] &\leq 
        (1+t_{\text{max}}) \mathbb{E}\left[ ||x_{k}-x^* ||^2+||\theta_k-\theta^* ||^2| \mathcal{F}_k \right] \nonumber \\
        &\hspace{3mm} - t_B \mathbb{E}\left[||x_k-x_{k+1} ||^2| \mathcal{F}_k \right] +2 u_k + t_C \gamma_k^2 + t_{F,k} ||\theta_k-\theta^* ||^2  \nonumber \\
        &\hspace{3mm}  +\beta_k^2 \nu^2 + 2\beta_k^2 \mathbb{E} \left[ (w_k)^T \left( \nabla h(\theta_k)-\nabla h(\theta^*) \right) | \mathcal{F}_k \right] \nonumber\\
        &\hspace{3mm} - 2\beta_k \mathbb{E} \left[ (w_k)^T \left( \theta_k - \theta^* \right) | \mathcal{F}_k \right],
    \end{align}
    where $t_{B}>0$, $t_{C}>0$, $u_k \leq 0$, $t_{\text{max}} = \max\{t_{A,k},t_{E,k}\}$, $t_{A,k} = \eta_1 \gamma_k^2+\eta_2 \gamma_k^{2-\tau}$, $t_{E,k}=\gamma_k^2 \left(\frac{2L_f^2 L_{\theta}^2}{\eta_1}+\frac{2L_{\theta}^2}{\eta_3} \right)$ and $t_{F,k}= \frac{L_{\theta}^2 \gamma_k^{\tau}}{\eta_2} - 2\beta_k \mu_h +\beta_k^2 L_h^2$. Specifically, if $t_{A,k}\geq t_{E,k}$, then 
    \begin{align*}
        \sum_{k=0}^{\infty} t_{\text{max}} = \sum_{k=0}^{\infty} t_{A,k}= \eta_1  \sum_{k=0}^{\infty} \gamma_k^2 + \eta_2 \sum_{k=0}^{\infty} \gamma_k^{2-\tau} <\infty;
    \end{align*}
    if $t_{E,k}> t_{A,k}$, then 
    \begin{align*}
        \sum_{k=0}^{\infty} t_{\text{max}}= \sum_{k=0}^{\infty} t_{E,k} =\left(\frac{2L_f^2 L_{\theta}^2}{\eta_1} + \frac{2L_{\theta}^2}{\eta_3} \right)\sum_{k=0}^{\infty} \gamma_k^2 <\infty.
    \end{align*}
    
    The last two terms in (\ref{ieq:condi_expectation_xk_theta_k}) are $0$ based on Assumption \ref{assumption:first_second_moment}. 
    As discussed at the end of Lemma~\ref{lemma:convergence_xk_theta_k}, for all $k \geq K$ with some $K$, it has $t_{F,k}\leq 0$. By rearranging (\ref{ieq:condi_expectation_xk_theta_k}), it gives 
    \begin{align*}
        \mathbb{E}\left[ ||x_{k+1}-x^* ||^2 + ||\theta_{k+1}-\theta^* ||^2 | \mathcal{F}_k \right] &\leq 
        (1+t_{\text{max}}) \mathbb{E}\left[ ||x_{k}-x^* ||^2+||\theta_k-\theta^* ||^2| \mathcal{F}_k \right] \nonumber \\
        &\hspace{3mm} - t_B \mathbb{E}\left[||x_k-x_{k+1} ||^2| \mathcal{F}_k \right] +2 u_k + t_{F,k} ||\theta_k-\theta^* ||^2 + t_C \gamma_k^2 +\beta_k^2 \nu^2 \nonumber \\
        &\leq (1+t_{\text{max}}) (||x_{k}-x^* ||^2+||\theta_k-\theta^* ||^2) +2 u_k + t_C \gamma_k^2 +\beta_k^2 \nu^2.
    \end{align*}
    Combining $\sum_{k=0}^{\infty} t_{\text{max}} < \infty$, $\sum_{k=0}^{\infty} \gamma_k^2 <\infty$, $\sum_{k=0}^{\infty} \beta_k^2 <\infty$ and Lemma \ref{lemma:convergence_as_2}, we conclude that $||x_k - x^*||^2\to 0$ almost surely as $k\to \infty$. 
\end{proof}

\subsection{Multi-Step Learning–Optimization Framework}

We will generalize the single-step joint update rule introduced in (\ref{alg:jointalg}) to a multi-step learning–optimization framework. This generalization allows multiple steps in both the learning and optimization phases. The ``multi-step" capability is important for balancing the convergence rates of the two levels, as learning tasks are often more time-consuming than decision tasks.

Formally, let $Q \geq 1$ and $R \geq 1$ denote the number of iteration counts for the outer-level variable $x$ and the inner-level variable $\theta$, respectively, within a single global iteration $k$. Note that $Q$ and $R$ can be different. We define the multi-step update rules as follows:
\begin{align} \label{formu:fusion_general}
x_{k}^{j+1/2} & = \Pi_{X} \left(x_k^{j}-\gamma_k \nabla f(x_k^j,\theta_k) \right), \quad j = 0....Q,   \\ 
x_{k}^{j+1} & = \Pi_{X} \left(x_k^j -\gamma_k \nabla f(x_{k}^{j+1/2},\theta_k) \right), \nonumber \\ 
\theta_{k}^{i+1} & = \Pi_{\Theta}(\theta_k^i-\beta_k (\nabla h(\theta_k^i)+\omega_{k}^i)),  \quad i = 0,..R, \nonumber\\
x_{k+1} & = x_{k+1}^{0} = x_k^Q, \nonumber\\
\theta_{k+1} & = \theta_{k+1}^{0} = \theta_k^R. \nonumber 
\end{align}
Here, the last two equations in (\ref{formu:fusion_general}) represent the starting states at iteration $k+1$. The sequences $\gamma_k$ and $\beta_k$ are the step sizes for the optimization and learning updates, respectively, and $w_k^i$ represents the stochastic noise in the learning gradient.

A convenient choice that satisfies Assumption~\ref{assumption:A1steps} is given by polynomial stepsize sequences
\begin{align*}
    \gamma_k = \frac{\gamma_0}{(k+1)^a}, \quad \beta_k = \frac{\beta_0}{(k+1)^b}, \quad \gamma_0, \beta_0 >0,
\end{align*}
where the parameters $a,b$ are selected as follows. The conditions on $\{\gamma_k\}$ are satisfied 
\begin{align*}
    0.5 < a \leq 1 \quad \text{and} \quad (2-\tau)a >1,
\end{align*}
while the conditions on $\{\beta_k\}$ require $0.5 < b \leq 1$. The limit condition $\frac{\gamma_k^{\tau}}{\beta_k} \to 0$ implies that $a\tau >b$. That is, given $\tau \in (0, 1)$ by Assumption~\ref{assumption:A1steps}, we need to make sure $(a, b, \tau)$ satisfies
\begin{align*}
    0.5 < b < a\tau < a \leq 1.
\end{align*}
A concrete example that is used in Algorithm~\ref{alg:fusion_alg} is $(a, b, \tau)=(1, 0.6, 0.75)$, i.e.,
\begin{align*}
    \gamma_k = \frac{\gamma_0}{k+1}, \quad \beta_k = \frac{\beta_0}{(k+1)^{0.6}}, \quad \gamma_0, \beta_0 >0.
\end{align*}

Algorithm~\ref{alg:fusion_alg} provides a pseudocode implementation of this multi-step learning–optimization algorithm. The algorithm explicitly defines the nested loop structure, where the outer-level loop executes $Q$ times to update the decision variables, and the inner-level loop performs $R$ times to update the learning parameters.

\begin{algorithm}[h]
\caption{Multi-Step Learning–Optimization Algorithm (MSLO)}
\label{alg:fusion_alg}
\textbf{Input:} Initial outer-level variable $x_0$, initial inner-level variable $\theta_{0}$, initial step sizes $\gamma_0$ and $\beta_0$, initial counts for outer and inner loop $Q$ and $R$, and stochastic noise $w_k$. \\
\textbf{Output:} Solutions $x_{k}$ and $\theta_{k}$.

\begin{algorithmic}[1]
\State Set the initial iterations $k=q=r=0$.
\While{a stopping criterion has not been met} 
 \State $\gamma_k \gets \frac{\gamma_0}{k+1}$
\State $\beta_k \gets \frac{\beta_0}{(k+1)^{0.6}}$
\While{$q<Q$} \label{alg:outer}  \Comment{Outer-level loop}
  \State $x_{k+1/2} \gets \Pi_{X} \left(x_k - \gamma_k \nabla f(x_k, \theta_k) \right)$ 
  \State $x_{k+1} \gets \Pi_{X} \left(x_k -\gamma_k \nabla f(x_{k+1/2}, \theta_k) \right)$
  \State $x_k \gets x_{k+1}$
  \State $q \gets q+1$
\EndWhile
\State Reset $q=0$
\While{$r<R$}  \Comment{Inner-level loop}
  \State $\theta_{k+1} \gets \Pi_{\Theta} \left(\theta_k -\beta_k \nabla h(\theta_{k} + w_k) \right)$
  \State $\theta_k \gets \theta_{k+1}$
  \State $r \gets r+1$
\EndWhile
\State Reset $r=0$
\State $k \gets k+1$
\EndWhile
\end{algorithmic}
\end{algorithm}

\subsubsection{Convergence Analysis for the Multi-Step Learning–Optimization Framework}

    We introduce two nested sequences of filtrations $\{\mathcal{F}_k\}_{k \geq 0}$ and $\{\mathcal{M}_k^i\}_{i=1}^R$ by
    \begin{align*}
        \mathcal{F}_k &:= \sigma\left( \{x_l, \theta_l\}_{l=0}^k \right) = \sigma(x_0, \theta_0, x_1, \theta_1, \dots, x_k, \theta_k), \\
        \mathcal{M}_k^i &:= \sigma\left( \mathcal{F}_k \cup \{w_k^{j-1}, \theta_k^j\}_{j=1}^i \right) = \sigma(\mathcal{F}_k, w_k^0, \theta_k^1, w_k^1, \dots, w_k^{i-1}, \theta_k^i).
    \end{align*}
    Given $\mathcal{F}_k$, the sequence $\{x_k^j\}_{j=0}^Q$ is generated deterministically by (\ref{formu:fusion_general}), which implies $\{x_k^j\}_{j=0}^Q \in \mathcal{F}_k$. Within the $k$-th global iteration, the stochastic components $\{w_k^{i-1}\}_{i=1}^{R}$ are captured by the filtration chain
    \begin{align*}
        \mathcal{F}_k \subseteq \mathcal{M}_k^0 \subseteq \mathcal{M}_k^1 \subseteq \dots \subseteq \mathcal{M}_k^R \subseteq \mathcal{F}_{k+1},
    \end{align*}
    which implies that the $i$-th inner iterate $\theta_k^i$ and its associated noise term satisfy $\theta_k^i \in \mathcal{M}_k^i$ and $w_k^{i-1} \in \mathcal{M}_k^i$ for all $i \in \{1, \dots, R\}$.
Moreover, it holds that
    \begin{align} \label{eq:w_k_condition_general}
        \mathbb{E} \left[w_k^i | \mathcal{F}_k \right] &= \mathbb{E}\left[ \mathbb{E}\left[ w_k^i | \mathcal{M}_{k}^i \right] | \mathcal{F}_{k} \right]= 0, \\
        \mathbb{E}\left[ ||w_k^i||^2 | \mathcal{F}_{k} \right] &= \mathbb{E}\left[ \mathbb{E}\left[ ||w_k^i||^2 | \mathcal{M}_{k}^i \right] | \mathcal{F}_{k} \right] \leq \mathbb{E}[\nu^2 | \mathcal{F}_k] = \nu^2. \nonumber
    \end{align}

\begin{proposition} \label{prop:theta_convergence_as}
    Given $\theta_k$ and $\mathcal{F}_k$. Let $R\geq 1$ be an integer. Then after $R$ times inner-level iterations, it holds $\theta_{k+1} = \theta_k^R$, and the sequence $\{\theta_{k+1}\} \to \theta^*$ almost surely, where $\theta^* = \argmin_{\theta \in \Theta} \hspace{1mm} h(\theta)$.
\end{proposition}
\begin{proof}
    First, we consider the case $R=1$.
    By Lemma \ref{lemma:convergence_xk_theta_k}, we know
    \begin{align*}
        ||\theta_{k+1} - \theta^*||^2 &\leq \left(1-2\beta_k\mu_h + \beta_k^2 L_h^2 \right) ||\theta_k-\theta^* ||^2 +\beta_k^2 ||w_k||^2 \\
        &\hspace{3mm} + 2\beta_k^2 (w_k)^T \left( \nabla h(\theta_k)-\nabla h(\theta^*) \right) - 2\beta_k (w_k)^T \left( \theta_k - \theta^* \right)
    \end{align*}
    for all $k$, where $\mu_h>0$ and $L_h >0$ is the Lipschitz constant of $\nabla h$. Taking conditional expectation with respect to $\mathcal{F}_k$, 
    it follows
    \begin{align} \label{ieq:condi_expectation_theta}
        \mathbb{E}\left[ ||\theta_{k+1} - \theta^*||^2 | \mathcal{F}_k \right] &\leq \left(1-2\beta_k\mu_h + \beta_k^2 L_h^2 \right) \mathbb{E}\left[ ||\theta_k-\theta^* ||^2 | \mathcal{F}_k \right] +\beta_k^2 \mathbb{E}\left[ ||w_k||^2 | \mathcal{F}_k \right] \nonumber\\
        &\hspace{3mm} + 2\beta_k^2 \mathbb{E} \left[ (w_k)^T \left( \nabla h(\theta_k)-\nabla h(\theta^*) \right) | \mathcal{F}_k \right]  - 2\beta_k \mathbb{E} \left[ (w_k)^T \left( \theta_k - \theta^* \right) | \mathcal{F}_k \right].
    \end{align}
    According to Assumption \ref{assumption:first_second_moment}, we have 
    \begin{align*}
        \mathbb{E}\left[ ||w_k||^2 | \mathcal{F}_k \right] <\nu^2, \quad 
        \mathbb{E}\left[ (w_k)^T \left( \nabla h(\theta_k)-\nabla h(\theta^*) \right) | \mathcal{F}_k \right] =0,
        \quad 
        \mathbb{E}\left[ (w_k)^T \left( \theta_k - \theta^* \right) | \mathcal{F}_k \right] =0.
    \end{align*}
    Thus, (\ref{ieq:condi_expectation_theta}) can be rewritten as 
    \begin{align*} 
        \mathbb{E}\left[ ||\theta_{k+1} - \theta^*||^2 | \mathcal{F}_k \right] &\leq \left(1- \underbrace{2\beta_k \mu_h + \beta_k^2 L_h^2}_{\alpha_k} \right) ||\theta_k-\theta^* ||^2 +\beta_k^2 \nu^2.
    \end{align*}
     Let $\alpha_k = 2\beta_k \mu_h - \beta_k^2 L_h^2$. Since $\beta_k \leq \beta_0 \leq \frac{2\mu_h}{L_h^2}$ by Assumption~\ref{assumption:A1steps}, it has $\alpha_k\geq 0$. Moreover, we can obtain $\alpha_k \leq 1$ since the condition $\mu_h \leq L_h$ holds. Therefore, we have $0\leq \alpha_k \leq 1$ for all $k$. Based on Assumption~\ref{assumption:A1steps}, it has
     \begin{align*}
         \sum_{k=0}^{\infty} \alpha_k = \sum_{k=0}^{\infty} \left(2\beta_k \mu_h - \beta_k^2 L_h^2 \right) =\infty, \quad \sum_{k=0}^{\infty} \beta_k^2\nu^2<\infty, \quad 
         \lim_{k\to \infty} \frac{\beta_k^2\nu^2}{\alpha_k} = \nu^2\lim_{k\to \infty} \frac{\beta_k }{2\mu - \beta_k L_h^2} =0.
     \end{align*}
    Applying Lemma \ref{lemma:convergence_as_1}, we obtain
    \begin{align*} 
        \lim_{k\to \infty} ||\theta_k - \theta^*||^2 =0 \hspace{2mm} \text{almost surely}.
    \end{align*}
    
    
    Next, we consider a more general case, i.e., $R> 1$. 
    Recall that $\theta_k=\theta_k^0$, $\theta_{k+1}=\theta_{k+1}^0=\theta_k^R$. Then we have 
    \begin{align} \label{ieq:theta_k_convergence_generalization}
        ||\theta_{k+1} - \theta^*||^2=||\theta_{k}^R - \theta^*||^2 
        &\leq (1 - \alpha_k) ||\theta_k^{R-1}-\theta^* ||^2 + \beta_k^2 ||w_k^{R-1}||^2 \nonumber\\
        &\hspace{3mm} + 2\beta_k^2 (w_k^{R-1})^T \left( \nabla h(\theta_k^{R-1})-\nabla h(\theta^*) \right) - 2\beta_k (w_k^{R-1})^T \left( \theta_k^{R-1} - \theta^* \right) \nonumber\\
         &\leq \left(1-\alpha_k \right)^2  ||\theta_k^{R-2}-\theta^* ||^2 + \left(1-\alpha_k \right)\beta_k^2 ||w_k^{R-2}||^2 \nonumber\\
         &\hspace{3mm} + 2 \left(1-\alpha_k \right) \beta_k^2 (w_k^{R-2})^T \left( \nabla h(\theta_k^{R-2})-\nabla h(\theta^*) \right) \nonumber \\
         &\hspace{3mm} - 2 \left(1-\alpha_k \right) \beta_k (w_k^{R-2})^T \left( \theta_k^{R-2} - \theta^* \right) + \beta_k^2 ||w_k^{R-1}||^2 \nonumber\\
         &\hspace{3mm}  + 2\beta_k^2 (w_k^{R-1})^T \left( \nabla h(\theta_k^{R-1})-\nabla h(\theta^*) \right) - 2\beta_k (w_k^{R-1})^T \left( \theta_k^{R-1} - \theta^* \right) \nonumber\\
         &\leq  \left(1-\alpha_k \right)^R ||\theta_k^{0}-\theta^* ||^2 + \beta_k^2 \sum_{i=0}^{R-1} \left( ||w_k^i||^2  \left(1-\alpha_k \right)^{R-1-i} \right)  \nonumber\\
         &\hspace{3mm} + 2 \beta_k^2 \sum_{i=0}^{R-1}\left( (w_k^i)^T \left( \nabla h(\theta_k^{i})-\nabla h(\theta^*) \right) \left(1-\alpha_k \right)^{R-1-i} \right) \nonumber\\
         &\hspace{3mm} -2 \beta_k \sum_{i=0}^{R-1}\left( (w_k^i)^T \left(\theta_k^i-\theta^* \right) \left(1-\alpha_k \right)^{R-1-i} \right).
    \end{align}

    Taking conditional expectation with respect to $\mathcal{F}_k$, 
    (\ref{ieq:theta_k_convergence_generalization}) gives
    \begin{align} \label{ieq:theta_convergence_R_greater1}
        \mathbb{E}\left[ ||\theta_{k}^R - \theta^*||^2 | \mathcal{F}_{k} \right] 
        &\leq \left(1-\alpha_k \right)^R ||\theta_k^{0}-\theta^* ||^2 + \beta_k^2 \sum_{i=0}^{R-1} \left( 1-\alpha_k \right)^{R-1-i} \mathbb{E}\left[ ||w_k^i||^2 | \mathcal{F}_{k} \right] \nonumber \\
        &\hspace{3mm} + 2 \beta_k^2 \sum_{i=0}^{R-1} \left(1-\alpha_k \right)^{R-1-i} \mathbb{E}\left[(w_k^i)^T \left( \nabla h(\theta_k^{i})-\nabla h(\theta^*) \right) | \mathcal{F}_{k} \right] \nonumber \\
        &\hspace{3mm} -2 \beta_k \sum_{i=0}^{R-1} \left(1-\alpha_k \right)^{R-1-i} \mathbb{E}\left[(w_k^i)^T \left(\theta_k^i-\theta^* \right) | \mathcal{F}_{k} \right].
    \end{align}

    Based on (\ref{eq:w_k_condition_general}) and using the tower property, we obtain
    \begin{align*} 
    \mathbb{E} \left[ ||w_k^i||^2 | \mathcal{F}_{k} \right] 
        &=\mathbb{E}\left[  \mathbb{E} \left[||w_k^i||^2 | \mathcal{M}_{k}^i \right] | \mathcal{F}_{k} \right] \leq \mathbb{E} \left[ \nu^2 | \mathcal{F}_{k} \right]=\nu^2, \\
        \mathbb{E}\left[(w_k^i)^T \left( \nabla h(\theta_k^{i})-\nabla h(\theta^*) \right) | \mathcal{F}_{k} \right] 
        &= \mathbb{E}\left[ \mathbb{E} \left[(w_k^i)^T \left( \nabla h(\theta_k^{i})-\nabla h(\theta^*) \right) | \mathcal{M}_{k}^i \right] | \mathcal{F}_{k} \right] = \mathbb{E}\left[ 0 | \mathcal{F}_{k} \right]=0, \nonumber \\
        \mathbb{E}\left[(w_k^i)^T \left(\theta_k^i-\theta^* \right) | \mathcal{F}_{k} \right] 
        &= \mathbb{E}\left[ \mathbb{E} \left[(w_k^i)^T \left(\theta_k^i-\theta^* \right) | \mathcal{M}_{k}^i \right] | \mathcal{F}_{k} \right] = \mathbb{E}\left[ 0 | \mathcal{F}_{k} \right]=0.
    \end{align*}
    Substituting above three expressions back into (\ref{ieq:theta_convergence_R_greater1}), and combining condition $0 \leq \alpha_k \leq 1$, it obtains
    \begin{align} 
        W_{2,k} :&= \beta_k^2 \sum_{i=0}^{R-1} \left( 1-\alpha_k \right)^{R-1-i} \mathbb{E}\left[ ||w_k^i||^2 | \mathcal{F}_{k} \right] \leq \beta_k^2 \nu^2 \sum_{i=0}^{R-1}  \left(1-\alpha_k \right)^{R-1-i} \leq R \beta_k^2 \nu^2 . \label{eq:w_2k} \\
        W_{1,k} :&=  2 \beta_k^2 \sum_{i=0}^{R-1} \left(1-\alpha_k \right)^{R-1-i} \mathbb{E}\left[(w_k^i)^T \left( \nabla h(\theta_k^{i})-\nabla h(\theta^*) \right) | \mathcal{F}_{k} \right] \nonumber \\
        &\hspace{3mm} -2 \beta_k \sum_{i=0}^{R-1} \left(1-\alpha_k \right)^{R-1-i} \mathbb{E}\left[(w_k^i)^T \left(\theta_k^i-\theta^* \right) | \mathcal{F}_{k} \right] \nonumber \\
        &=0 , \label{eq:w_1k}
    \end{align}
    
     Then, we can simplifies (\ref{ieq:theta_convergence_R_greater1}) to
    \begin{align*}
        \mathbb{E}\left[ ||\theta_{k}^R - \theta^*||^2 | \mathcal{F}_{k} \right] 
        &< \left(1-\alpha_k \right)^R||\theta_k^{0}-\theta^* ||^2 + R \beta_k^2 \nu^2.   
    \end{align*}

Since $(1-x)^n \leq \frac{1}{1+nx}$ with $0\leq x\leq 1$ and $n\in \mathbb{N}$ and $0 \leq \alpha_k \leq 1$ for all $k$, it is natural to get
    \begin{align} \label{ieq:psi_relationship}
        \left(1-\alpha_k  \right)^R 
        \leq  \frac{1}{1+R\alpha_k } 
        = 1- \frac{R \alpha_k }{1+ R \alpha_k }   
        \leq 1 - \frac{R \alpha_k}{1+R}.
    \end{align}
    
    Next, we will check if the following four conditions hold:
    \begin{align*}
        0\leq \frac{R \alpha_k }{1+ R} \leq 1, \quad 
        \sum_{k=0}^{\infty} \frac{R \alpha_k }{1+ R } = \infty, \quad  
        R\nu^2 \sum_{k=0}^{\infty}  \beta_k^2 <\infty, \quad
        \lim_{k \to \infty}\frac{R \beta_k^2 \nu^2 ( 1+R)}{ R \alpha_k } =0.
    \end{align*}
    \begin{itemize}
    \item For any $k$, it has $0 \leq \alpha_k \leq1$. Then, it is easy to have $0 \leq \frac{R \alpha_k }{1+ R } < 1$. Moreover,
    \begin{align*}
        \sum_{k=0}^{\infty} \frac{R \alpha_k }{1+ R } = \frac{R}{1+R} \sum_{k=0}^{\infty} \alpha_k = \frac{R}{1+R} \sum_{k=0}^{\infty} \left(2\beta_k \mu_h - \beta_k^2 L_h^2 \right) = \infty.
    \end{align*}

    \item Based on Assumption~\ref{assumption:A1steps}, it has $ R \nu^2 \sum_{k=0}^{\infty}  \beta_k^2 < \infty$.

    \item Lastly, since $\lim_{k \to \infty} \beta_k=0$ according to Assumption~\ref{assumption:A1steps}, 
   we can obtain
   \begin{align*}
   \lim_{k \to \infty}\frac{\beta_k^2 \nu^2 (1+R)}{\alpha_k} 
        = \nu^2 (1+R) \lim_{k\to \infty} \frac{\beta_k^2}{\alpha_k} = \nu^2 (1+R) \lim_{k\to \infty} \frac{\beta_k}{2\mu - \beta_k L_h^2} =0.
    \end{align*}
    \end{itemize}

    Thus, we conclude that $\lim_{k\to \infty} ||\theta_{k+1} - \theta^*||=0$ almost surely by Lemma~\ref{lemma:convergence_as_1}.

\end{proof}

\begin{proposition} \label{prop:xk_theta_k_convergence_as_generalization}
    Given $x_k$, $\theta_k$ and $\mathcal{F}_k$. The algorithm alternates between $Q \geq 1$ outer-level optimization steps and $R \geq 1$ inner-level learning updates. Then it holds that the sequence $\{x_{k+1}\}$ converges almost surely to a point $x^*$ that satisfies the first-order (KKT) stationarity conditions for the problem $\min_{x \in X} f(x, \theta^*)$.
\end{proposition}
\begin{proof}
    Let the sequences $\{x_k\}$ and $\{\theta_k\}$ be generated by the algorithm~\ref{alg:fusion_alg}. 
    For any $j\in \{0,...,Q\}$, we assume $u_k^j= \gamma_k (x^*-x_k^j)^T \nabla f(x_k^j,\theta^*)$.
    Note that
    \begin{align*}
        x_k=x_k^0, \quad x_{k+1}=x_{k+1}^0=x_k^Q, \\
        \theta_k = \theta_k^0, \quad \theta_{k+1}=\theta_{k+1}^0 = \theta_k^R. 
    \end{align*}
    
    Then we will discuss the convergence property case by case.
    \begin{enumerate}
        \item If $Q=R=1$, it is equivalent to the single-step case in Section~\ref{sec:single_step_alg}. Therefore, Proposition~\ref{prop:xk_theta_k_convergence_as} gives the result.
    
        \item If $Q>1$ and $R=1$, then (\ref{eq:xk_bounded}) becomes
     \begin{align} \label{ieq:xk_bound_generalization_1}    
     ||x_{k+1}-x^* ||^2 = ||x_{k}^Q-x^* ||^2 
     &\leq \left( 1+t_{A,k}\right) ||x_k^{Q-1}-x^* ||^2 - t_B ||x_{k}^{Q-1}-x_{k}^{Q} ||^2  \nonumber\\ 
    &\hspace{3mm} +2 u_k^{Q-1} + t_C \gamma_k^2 + t_{D,k} ||\theta_k-\theta^* ||^2 \nonumber\\ 
    &\leq \left( 1+t_{A,k}\right)^2 ||x_k^{Q-2}-x^* ||^2 - t_B \left( 1+t_{A,k}\right) ||x_{k}^{Q-2}-x_{k}^{Q-1} ||^2 \nonumber\\ 
    &\hspace{3mm} +2 \left( 1+t_{A,k}\right) u_k^{Q-2} + t_C \gamma_k^2 \left( 1+t_{A,k}\right)  + t_{D,k} \left( 1+t_{A,k}\right) ||\theta_k-\theta^* ||^2  \nonumber\\ 
   &\hspace{3mm} - t_B ||x_{k}^{Q-1}-x_{k}^{Q} ||^2 + 2 u_k^{Q-1} + t_C \gamma_k^2 + t_{D,k} ||\theta_k-\theta^* ||^2 \nonumber\\ 
   &\leq  \left( 1+t_{A,k}\right)^Q ||x_k^0-x^*||^2- t_B \sum_{j=0}^{Q-1}\left(||x_k^{j+1}-x_k^{j}||^2 \left( 1+t_{A,k}\right)^{Q-1-j} \right) \nonumber\\ 
   &\hspace{3mm} + 2\left(\sum_{j=0}^{Q-1} u_k^j \left( 1+t_{A,k}\right)^{Q-1-j} \right) + t_C \gamma_k^2 \left( \sum_{j=0}^{Q-1}  \left( 1+t_{A,k}\right)^{Q-1-j} \right) \nonumber\\ 
   &\hspace{3mm} + t_{D,k} \left(\sum_{j=0}^{Q-1}  \left( 1+t_{A,k}\right)^{Q-1-j} \right) ||\theta_k-\theta^*||^2,
    \end{align}
    where $t_{B}>0$, $t_{C}>0$, $u_k^j \leq 0$ for all $k$ and $j$, $t_{A,k} = \eta_1 \gamma_k^2+\eta_2 \gamma_k^{2-\tau}$, $t_{D,k} = t_{E,k} + \frac{L_{\theta}^2 \gamma_k^{\tau}}{\eta_2}= \gamma_k^2 \left( \frac{2L_f^2 L_{\theta}^2}{\eta_1} + \frac{2L_{\theta}^2}{\eta_3} \right)+ \frac{L_{\theta}^2 \gamma_k^{\tau}}{\eta_2}$.
   Combining (\ref{ieq:xk_bound_generalization_1}) and Proposition~\ref{prop:theta_convergence_as}, it gives
    \begin{align} \label{ieq:xk_theta_k_bound_generalization}
    ||x_{k+1}-x^* ||^2+||\theta_{k+1}-\theta^*||^2 
    &\leq  \left( 1+t_{A,k}\right)^Q ||x_k^0-x^*||^2- t_B \sum_{j=0}^{Q-1}\left(||x_k^{j+1}-x_k^{j}||^2 \left( 1+t_{A,k}\right)^{Q-1-j} \right) \nonumber\\ 
   &\hspace{3mm} + 2\left(\sum_{j=0}^{Q-1} u_k^j \left( 1+t_{A,k}\right)^{Q-1-j} \right) + t_C \gamma_k^2 \left( \sum_{j=0}^{Q-1}  \left( 1+t_{A,k}\right)^{Q-1-j} \right) \nonumber\\ 
   &\hspace{3mm} + \left(t_{D,k} \sum_{j=0}^{Q-1} \left( 1+t_{A,k}\right)^{Q-1-j} + 1-\alpha_k \right) ||\theta_k-\theta^*||^2 +\beta_k^2 ||w_k||^2 \nonumber\\
   &\hspace{3mm} + 2\beta_k^2 (w_k)^T \left( \nabla h(\theta_k)-\nabla h(\theta^*) \right) - 2\beta_k (w_k)^T \left( \theta_k - \theta^* \right).
    \end{align}
    Given $\alpha_k = 2\beta_k \mu_h - \beta_k^2 L_h^2$. Then we have
    \begin{align} \label{eq:t_gk^Q_and_t_hk^Q}
         t_{D,k} \sum_{j=0}^{Q-1} \left( 1+t_{A,k}\right)^{Q-1-j} + 1-\alpha_k 
        &= \left( t_{E,k} +\frac{L_{\theta}^2 \gamma_k^{\tau}}{\eta_2} \right) \sum_{j=0}^{Q-1}  \left( 1+t_{A,k}\right)^{Q-1-j} + 1-\alpha_k \nonumber \\
        &= \underbrace{1+ t_{E,k}\sum_{j=0}^{Q-1}   \left( 1+t_{A,k}\right)^{Q-1-j}}_{t_{G,k}^Q} \nonumber\\
        &\hspace{3mm} + \underbrace{\frac{L_{\theta}^2 \gamma_k^{\tau}}{\eta_2} \sum_{j=0}^{Q-1} \left( 1+t_{A,k}\right)^{Q-1-j} -2\beta_k \mu_h + \beta_k^2 L_h^2}_{t_{H,k}^Q}.
    \end{align}
    
    Taking conditional expectation with respect to $\mathcal{F}_k$, (\ref{ieq:xk_theta_k_bound_generalization}) gives
    \begin{align*} 
        \mathbb{E}\left[||x_{k+1}-x^* ||^2+||\theta_{k+1}-\theta^*||^2| \mathcal{F}_{k} \right] &\leq  \left( 1+t_{A,k}\right)^Q ||x_k^0-x^*||^2 \nonumber\\
        &\hspace{3mm} - t_B \sum_{j=0}^{Q-1} \left( 1+t_{A,k}\right)^{Q-1-j} \mathbb{E}\left[||x_k^{j+1}-x_k^{j}||^2| \mathcal{F}_{k} \right] \nonumber\\
        &\hspace{3mm} + 2\sum_{j=0}^{Q-1} \left( 1+t_{A,k}\right)^{Q-1-j} \mathbb{E}\left[u_k^j| \mathcal{F}_{k} \right] + t_C \gamma_k^2 \sum_{j=0}^{Q-1}  \left( 1+t_{A,k}\right)^{Q-1-j}  \nonumber\\
        &\hspace{3mm} + (t_{G,k}^Q + t_{H,k}^Q) ||\theta_k-\theta^*||^2 + \beta_k^2 \mathbb{E}\left[||w_k||^2 |\mathcal{F}_{k} \right] \nonumber\\
        &\hspace{3mm} + 2\beta_k^2 \mathbb{E} \left[ (w_k)^T \left( \nabla h(\theta_k)-\nabla h(\theta^*) \right) | \mathcal{F}_k \right] \nonumber\\
        &\hspace{3mm} - 2\beta_k \mathbb{E} \left[ (w_k)^T \left( \theta_k - \theta^* \right) | \mathcal{F}_k \right] \nonumber\\
        &\leq (1+t_{J,k}^Q) \left(||x_k^0 -x^* ||^2+||\theta_{k}-\theta^*||^2\right) \nonumber\\
        &\hspace{3mm} - t_B \sum_{j=0}^{Q-1} \left( 1+t_{A,k}\right)^{Q-1-j} \mathbb{E}\left[||x_k^{j+1}-x_k^{j}||^2| \mathcal{F}_{k} \right] \nonumber\\
        &\hspace{3mm} + 2\sum_{j=0}^{Q-1} \left( 1+t_{A,k}\right)^{Q-1-j} \mathbb{E}\left[u_k^j| \mathcal{F}_{k} \right]  + t_C \gamma_k^2  \sum_{j=0}^{Q-1} \left( 1+t_{A,k}\right)^{Q-1-j} \nonumber\\
        &\hspace{3mm} +t_{H,k}^Q ||\theta_{k}-\theta^*||^2 +\beta_k^2 \nu^2,
    \end{align*}
    where $t_{J,k}^Q=\max \{\left( 1+t_{A,k}\right)^Q, \hspace{1mm}  t_{G,k}^Q \}-1$. 

    Given the filtration $\mathcal{F}_k$, we have $(x_k, \theta_k) \in \mathcal{F}_k$ with $x_k=x_k^0$. Since the sequence $\{x_k^j\}_{j=1}^Q$ is generated deterministically by (\ref{formu:fusion_general}) without additional noise, it follows that
    \begin{align*}
        x_k^j \in \mathcal{F}_k, \quad \forall j \in \{1, \dots, Q\}
    \end{align*}
    Consequently, by the property of conditional expectation for $\mathcal{F}_k$-measurable random variables, we have
    \begin{align*}
        \mathbb{E}\left[||x_k^{j+1}-x_k^{j}||^2| \mathcal{F}_{k} \right] = ||x_k^{j+1}-x_k^{j}||^2 <\infty.
    \end{align*}
   
    Next, we will check if the following four conditions hold: 
    \begin{align*}
        &\sum_{k=0}^{\infty} t_{J,k}^Q <\infty, \quad 
        \lim_{k \to \infty} t_{H,k}^Q\leq 0, \quad
        t_C \gamma_k^2 \sum_{j=0}^{Q-1} \left( 1+t_{A,k}\right)^{Q-1-j} +\beta_k^2 \nu^2 >0,\\ 
        &\sum_{k=0}^{\infty}  \left(t_C \gamma_k^2 \sum_{j=0}^{Q-1} \left( 1+t_{A,k}\right)^{Q-1-j} +\beta_k^2 \nu^2\right) <\infty.
    \end{align*}
    
    \begin{itemize}
        \item  
        We know from Assumption~\ref{assumption:A1steps} that $\sum_{k=0}^{\infty} t_{A,k} = \eta_1 \sum_{k=0}^{\infty} \gamma_k^2 + \eta_2 \sum_{k=0}^{\infty} \gamma_k^{2-\tau} < \infty$ and $t_{A,k} \to 0$ as $k \to \infty$. 
        Hence, there exists some integer $K_0$ such that for all $k \ge K_0$, it holds $t_{A,k} \leq 1$ and
    \begin{align} \label{ieq:t_ak_bound}
        \sum_{j=0}^{Q-1} \left( 1+t_{A,k}\right)^{Q-1-j} \leq \sum_{j=0}^{Q-1} \left( 1+t_{A,k}\right)^{Q-1} \leq Q \left( 1+t_{A,k}\right)^{Q-1} \leq Q 2^{Q-1}.
    \end{align}
        
        If $\left( 1+t_{A,k}\right)^Q \geq t_{G,k}^Q$, then for all $k \geq K_0$, by the algebraic factorization and $t_{A,k} \leq 1$, we obtain
        \begin{align*}
            t_{J,k}^Q = (1+ t_{A,k})^Q - 1^Q &= ((1+t_{A,k}) - 1) \sum_{j=0}^{Q-1} (1+t_{A,k})^j (1)^{Q-1-j} \\
            &= t_{A,k} \sum_{j=0}^{Q-1} (1+t_{A,k})^j \leq t_{A,k} \sum_{j=0}^{Q-1} 2^j = \left(2^Q - 1 \right) t_{A,k}.
        \end{align*}
    That leads to
    \begin{align} \label{ieq:rk_2n_finite_1}
    \sum_{k=K_0}^{\infty} t_{J,k} \leq \left(2^Q - 1 \right) \sum_{k=K_0}^{\infty} t_{A,k} <\infty \quad \text{and} \quad
        \sum_{k=0}^{\infty} t_{J,k}^Q = \sum_{k=0}^{K_0 -1} t_{J,k}^Q + \sum_{k=K_0}^{\infty} t_{J,k}^Q <\infty.
    \end{align} 
     The term $\sum_{k=0}^{K_0-1} t_{J,k}^Q$ is finite since it has finitely many terms. 
     
     If $t_{G,k}^Q >  \left( 1+t_{A,k}\right)^Q$, then for all $k \geq K_0$ and by (\ref{ieq:t_ak_bound}), it gives
    \begin{align} \label{ieq:rk_2n_finite_2}
        &t_{J,k}^Q = t_{G,k}^Q -1 
        = t_{E,k} \sum_{j=0}^{Q-1} \left( 1+t_{A,k}\right)^{Q-1-j} 
        \leq  t_{E,k} Q 2^{Q-1}. \nonumber\\
 \text{and} \quad   &\sum_{k=0}^{\infty} t_{J,k}^Q = \sum_{k=0}^{K_0-1} t_{J,k}^Q +                 \sum_{k=K_0}^{\infty} t_{J,k}^Q \nonumber \\
        &\hspace{13mm} \leq \sum_{k=0}^{K_0-1} t_{J,k}^Q + Q 2^{Q-1}\sum_{k=K_0}^{\infty} t_{E,k} \nonumber \\
        &\hspace{13mm}= \sum_{k=0}^{K_0-1} t_{J,k}^Q + Q 2^{Q-1} \left( \frac{2L_f^2 L_{\theta}^2}{\eta_1} + \frac{2L_{\theta}^2}{\eta_3} \right)\sum_{k=K_0}^{\infty} \gamma_k^2 <\infty.
    \end{align}
    Similarly, the term $\sum_{k=0}^{K_0-1} t_{J,k}^Q$ is finite. Thus, we conclude that $\sum_{k=0}^{\infty} t_{J,k}^Q<\infty$ holds. 
    
    \item Since $t_C>0$ and $t_{A,k}>0$ for all $k$, it is clear that 
    \begin{align} \label{eq:case2_t_c_gamma_k^2_1}
        t_C \gamma_k^2 \sum_{j=0}^{Q-1} \left( 1+t_{A,k}\right)^{Q-1-j} +\beta_k^2 \nu^2 >0.
    \end{align}
    For all $k \geq K_0$ and by (\ref{ieq:t_ak_bound}), we obtain 
    \begin{align} \label{eq:case2_t_c_gamma_k^2_2}
        &\hspace{3mm}\sum_{k=0}^{\infty}  \left(t_C \gamma_k^2 \sum_{j=0}^{Q-1} \left( 1+t_{A,k}\right)^{Q-1-j} +\beta_k^2 \nu^2\right) \nonumber \\
        &= t_C \sum_{k=0}^{K_0-1} \left( \gamma_k^2 \sum_{j=0}^{Q-1} \left( 1+t_{A,k}\right)^{Q-1-j} \right) + t_C \sum_{k=K_0}^{\infty} \left(\gamma_k^2 \sum_{j=0}^{Q-1} \left( 1+t_{A,k}\right)^{Q-1-j} \right )+ \nu^2 \sum_{k=0}^{\infty} \beta_k^2 \nonumber\\
        &\leq t_C \sum_{k=0}^{K_0-1} \left( \gamma_k^2 \sum_{j=0}^{Q-1} \left( 1+t_{A,k}\right)^{Q-1-j} \right) + t_C Q2^{Q-1} \sum_{k=K_0}^{\infty} \gamma_k^2 + \nu^2 \sum_{k=0}^{\infty} \beta_k^2 <\infty.
    \end{align}
    
    \item Lastly, let us consider the term $t_{H,k}^Q$. Since $\beta_k >0$ for all $k$, then 
    for any $k \geq K_0$ and by (\ref{ieq:t_ak_bound}), it has
    \begin{align*}
        \limsup_{k\to \infty} \left(\frac{t_{H,k}^Q}{\beta_k} \right)
        &= \limsup_{k\to \infty} \left(\frac{L_{\theta}^2 \gamma_k^{\tau}}{\eta_2 \beta_k} \sum_{j=0}^{Q-1} \left( 1+t_{A,k}\right)^{Q-1-j} -2 \mu_h + \beta_k L_h^2 \right) \\
        &\leq \frac{L_{\theta}^2 Q 2^{Q-1}}{\eta_2} \lim_{k\to \infty} \frac{ \gamma_k^{\tau}}{\beta_k} - 2 \mu_h + L_h^2 \lim_{k\to \infty} \beta_k =-2\mu_h < 0.
    \end{align*}
    Then, there exists $K_1$ such that for all $k \geq \max\{K_0, K_1\}$, we have $\frac{t_{H,k}^Q}{\beta_k} \leq - \mu_h$, and it leads to $t_{H,k}^Q \leq - \mu_h \beta_k<0$.
    That is, there exists an integer $\max\{K_0, K_1\}$ such that for every $k \geq \max\{K_0, K_1\}$, we have $t_{H,k}^Q < 0$. 
    \end{itemize}
     Therefore, we can conclude that the sequence $\{x_{k+1}\}$ converges to $x^*$ almost surely for case 2 (i.e., $Q>1,R=1$) by Lemma~\ref{lemma:convergence_as_2}.

    \item If $Q=1$ and $R>1$, then
    combining Lemma~\ref{lemm:xiter} and Proposition~\ref{prop:theta_convergence_as}, it is easy to have
    \begin{align} \label{ieq:xk_theta_k_Q1_R_general}
        ||x_{k+1}-x^*||+||\theta_{k+1} - \theta^*||^2 &= ||x_{k+1}-x^*|| + ||\theta_{k}^R - \theta^*||^2 \nonumber \\
        &\leq \left( 1+t_{A,k} \right) ||x_k - x^* ||^2 - t_B ||x_{k}-x_{k+1} ||^2 \nonumber\\ 
        &\hspace{3mm} + 2u_k  + \left( t_{D,k} +\left(1-\alpha_k  \right)^R \right) ||\theta_k^0 -\theta^* ||^2 \nonumber\\
        &\hspace{3mm} + t_C \gamma_k^2 + \beta_k^2 \sum_{i=0}^{R-1} ||w_k^i||^2 \left(1-\alpha_k \right)^{R-1-i} \nonumber\\
        &\hspace{3mm} + 2 \beta_k^2 \sum_{i=0}^{R-1}\left( (w_k^i)^T \left( \nabla h(\theta_k^{i})-\nabla h(\theta^*) \right) \left(1-\alpha_k \right)^{R-1-i} \right) \nonumber\\
        &\hspace{3mm} -2 \beta_k \sum_{i=0}^{R-1}\left( (w_k^i)^T \left(\theta_k^i-\theta^* \right) \left(1-\alpha_k \right)^{R-1-i} \right),
    \end{align}
    where $t_{D,k} = t_{E,k} + \frac{L_{\theta}^2 \gamma_k^{\tau}}{\eta_2}= \gamma_k^2 \left( \frac{2L_f^2 L_{\theta}^2}{\eta_1} + \frac{2L_{\theta}^2}{\eta_3} \right)+ \frac{L_{\theta}^2 \gamma_k^{\tau}}{\eta_2}$. 
    We know $\alpha_k=2\beta_k\mu_h - \beta_k^2 L_h^2$ and $\left(1-\alpha_k  \right)^R \leq 1- \frac{R \alpha_k }{1+ R}$ from (\ref{ieq:psi_relationship}), it gives
    \begin{align*}
    t_{D,k} + \left(1-\alpha_k  \right)^R 
    \leq t_{E,k} + \frac{L_{\theta}^2 \gamma_k^{\tau}}{\eta_2} +1- \frac{R\alpha_k }{1+R} = t_{E,k} + 1 + \underbrace{\frac{L_{\theta}^2 \gamma_k^{\tau}}{\eta_2} - \frac{R}{1+R} \left( 2\beta_k\mu_h - \beta_k^2 L_h^2\right)}_{t_{L,k}^R}.    
    \end{align*}
    Taking conditional expectation with respect to $\mathcal{F}_k$, (\ref{ieq:xk_theta_k_Q1_R_general}) gives
    \begin{align} \label{ieq:xk_theta_k_Q1_R_general_expectation}
        \mathbb{E}\left[||x_{k+1}-x^* ||^2+||\theta_{k+1}-\theta^*||^2| \mathcal{F}_{k} \right] 
        &\leq (1 + t_{\text{max}}) (||x_k-x^*||^2+||\theta_k-\theta^*||^2)    \nonumber \\
        &\hspace{3mm} - t_B ||x_{k}-x_{k+1} ||^2 + 2u_k + t_{L,k}^R ||\theta_k^0-\theta^*||^2 + t_C \gamma_k^2 \nonumber \\
        &\hspace{3mm} + \beta_k^2 \sum_{i=0}^{R-1}  \left(1-\alpha_k \right)^{R-1-i}  \mathbb{E}\left[ ||w_k^i||^2 | \mathcal{F}_{k} \right] \nonumber \\
        &\hspace{3mm} + 2 \beta_k^2 \sum_{i=0}^{R-1} \left(1-\alpha_k \right)^{R-1-i} \mathbb{E}\left[(w_k^i)^T \left( \nabla h(\theta_k^{i})-\nabla h(\theta^*) \right) | \mathcal{F}_{k} \right]  \nonumber \\
        &\hspace{3mm} -2 \beta_k \sum_{i=0}^{R-1} \left(1-\alpha_k \right)^{R-1-i} \mathbb{E}\left[(w_k^i)^T \left(\theta_k^i-\theta^* \right) | \mathcal{F}_{k} \right], \nonumber \\
        &= (1 + t_{\text{max}}) (||x_k-x^*||^2+||\theta_k-\theta^*||^2) - t_B ||x_{k}-x_{k+1} ||^2   \nonumber \\
        &\hspace{3mm}  + 2u_k + t_{L,k}^R ||\theta_k^0-\theta^*||^2 + t_C \gamma_k^2 + W_{2,k} + W_{1,k}
    \end{align}
    where $t_{\text{max}}=\max \{t_{A,k}, t_{E,k} \}$, $W_{1,k}$ and $W_{2,k}$ are defined as in (\ref{eq:w_1k}) and (\ref{eq:w_2k}). 
     Since $\mu_h >0$, we can obtain
    \begin{align*}
        \lim_{k\to \infty} \left(\frac{t_{L,k}^R}{\beta_k} \right)
        = \frac{L_{\theta}^2}{\eta_2} \lim_{k\to \infty} \frac{ \gamma_k^{\tau}}{\beta_k} - \frac{2\mu_h R}{1+R} + \frac{R L_h^2}{1+R} \lim_{k\to \infty} \beta_k = -\frac{2\mu_h R}{1+R} < 0. 
    \end{align*}
    Hence, there exists $K_2$ such that for all $k \geq K_2$, one has $\frac{t_{L,k}^R}{\beta_k} \leq - \frac{\mu_h}{1+R}$, which implies
    \begin{align*}
        t_{L,k}^R \leq - \frac{\mu_h \beta_k}{1+R} <0
    \end{align*}
    for all $k \geq K_2$. 
    According to (\ref{eq:w_2k}) and (\ref{eq:w_1k}), we know $W_{1,k}=0$ and $0< W_{2,k} < R \beta_k^2 \nu^2$. Then (\ref{ieq:xk_theta_k_Q1_R_general_expectation}) can be simplified as
     \begin{align*}
        \mathbb{E}\left[||x_{k+1}-x^* ||^2+||\theta_{k+1}-\theta^*||^2| \mathcal{F}_{k} \right] 
        &\leq (1+ t_{\text{max}}) (||x_k-x^*||^2+||\theta_k-\theta^*||^2) + 2u_k\\
        &\hspace{3mm} + t_C \gamma_k^2 + R \beta_k^2 \nu^2.
    \end{align*}
    Next, we will check if the two conditions in Lemma~\ref{lemma:convergence_as_2} are satisfied: 
    \begin{align*}
        \sum_{k=0}^{\infty} t_{\text{max}} <\infty \quad \text{and} \quad \sum_{k=0}^{\infty} \left( t_C \gamma_k^2 +  R \beta_k^2 \nu^2 \right) <\infty.
    \end{align*}

    \begin{itemize}
        \item Since $\sum_{k=0}^{\infty} \gamma_k^2 <\infty$, $\sum_{k=0}^{\infty} \beta_k^2 <\infty$ by Assumption~\ref{assumption:A1steps} and $\sum_{k=0}^{\infty} t_{A,k} < \infty$, it has
        \begin{align*}
            &\sum_{k=0}^{\infty} t_{\text{max}} =\sum_{k=0}^{\infty} \max\{t_{A,k}, t_{E,k}\}= \sum_{k=0}^{\infty} \max \left\{t_{A,k}, \left(\frac{2 L_f^2 L_{\theta}^2}{\eta_1}+\frac{2 L_{\theta}^2}{\eta_3} \right)\gamma_k^2 \right\} < \infty, \\
\text{and} \quad        &\sum_{k=0}^{\infty} \left( t_C \gamma_k^2 +  R \beta_k^2 \nu^2 \right) = t_C \sum_{k=0}^{\infty} \gamma_k^2 + R \nu^2 \sum_{k=0}^{\infty} \beta_k^2  <\infty.
        \end{align*}
    \end{itemize}
    Both conditions are met, therefore, it holds that $x_k \to x^*$ as $k\to \infty$ for case 3 (i.e., $Q=1,R>1$) by Lemma~\ref{lemma:convergence_as_2}. 
    
    \item If $Q>1$ and $R>1$, where $Q$ and $R$ could be different, then combining (\ref{ieq:theta_k_convergence_generalization}) and (\ref{ieq:xk_bound_generalization_1}), it gives
    \begin{align} \label{ieq:convergence_c4_QR_generalization}
        ||x_{k+1}-x^* ||^2+||\theta_{k+1}-\theta^*||^2 &= ||x_{k}^Q-x^* ||^2+||\theta_{k}^R-\theta^*||^2 \nonumber \\
        &\leq  \left( 1+t_{A,k}\right)^Q ||x_k^0-x^*||^2- t_B \sum_{j=0}^{Q-1}\left(||x_k^{j+1}-x_k^{j}||^2 \left( 1+t_{A,k}\right)^{Q-1-j} \right) \nonumber\\ 
        &\hspace{3mm} + 2\left(\sum_{j=0}^{Q-1} u_k^j \left( 1+t_{A,k}\right)^{Q-1-j} \right) + t_C \gamma_k^2 \sum_{j=0}^{Q-1}  \left( 1+t_{A,k}\right)^{Q-1-j}  \nonumber\\ 
        &\hspace{3mm} + \left(t_{D,k} \sum_{j=0}^{Q-1} \left( 1+t_{A,k}\right)^{Q-1-j} + \left(1-\alpha_k \right)^R \right) ||\theta_k^0 -\theta^*||^2 \nonumber\\
        &\hspace{3mm} + \beta_k^2 \sum_{i=0}^{R-1} \left(||w_k^i||^2 \left(1-\alpha_k \right)^{R-1-i}\right)  \nonumber\\
        &\hspace{3mm} + 2 \beta_k^2 \sum_{i=0}^{R-1}\left( (w_k^i)^T \left( \nabla h(\theta_k^{i})-\nabla h(\theta^*) \right) \left(1-\alpha_k \right)^{R-1-i} \right) \nonumber\\
        &\hspace{3mm} -2 \beta_k \sum_{i=0}^{R-1}\left( (w_k^i)^T \left(\theta_k^i-\theta^* \right) \left(1-\alpha_k \right)^{R-1-i} \right). 
    \end{align}

    Given $\alpha_k=2\beta_k \mu_h - \beta_k^2 L_h^2$ and $\left(1-\alpha_k  \right)^R \leq 1- \frac{R \alpha_k }{1+ R}$ by (\ref{ieq:psi_relationship}), we obtain
    \begin{align*}
        t_{D,k}  \sum_{j=0}^{Q-1} \left( 1+t_{A,k}\right)^{Q-1-j} + \left(1-\alpha_k \right)^R 
        &= \left( t_{E,k} + \frac{L_{\theta}^2 \gamma_k^{\tau}}{\eta_2} \right) \sum_{j=0}^{Q-1} \left( 1+t_{A,k}\right)^{Q-1-j} + \left(1-\alpha_k \right)^R \\
        &\leq \underbrace{t_{E,k} \sum_{j=0}^{Q-1} \left( 1+t_{A,k}\right)^{Q-1-j}  + 1}_{t_{G,k}^Q}  \\
        &\hspace{3mm} + \underbrace{\frac{L_{\theta}^2 \gamma_k^{\tau}}{\eta_2} \sum_{j=0}^{Q-1} \left( 1+t_{A,k}\right)^{Q-1-j} - \frac{R}{1+R} \left( 2\beta_k \mu_h - \beta_k^2 L_h^2\right)}_{t_{N,k}^{QR}}.
    \end{align*}
    Note that the term $t_{G,k}^Q$ also appears in (\ref{eq:t_gk^Q_and_t_hk^Q}).
    Taking conditional expectation with respect to $\mathcal{F}_k$ to  (\ref{ieq:convergence_c4_QR_generalization}), we get
    \begin{align*} 
        \mathbb{E}\left[||x_{k+1}-x^* ||^2+||\theta_{k+1}-\theta^*||^2| \mathcal{F}_{k} \right] 
        &\leq (1+t_{M,k}^{QR}) (||x_k^0 -x^*||^2+||\theta_k^0 -\theta^*||^2)  \nonumber\\
        &\hspace{3mm} - t_B \sum_{j=0}^{Q-1} \left( 1+t_{A,k}\right)^{Q-1-j} \mathbb{E}\left[||x_k^{j+1}-x_k^{j}||^2| \mathcal{F}_{k} \right] \nonumber\\
        &\hspace{3mm} + 2\sum_{j=0}^{Q-1} \left( 1+t_{A,k}\right)^{Q-1-j} \mathbb{E}\left[u_k^j| \mathcal{F}_{k} \right] +t_C \gamma_k^2 \sum_{j=0}^{Q-1} \left( 1+t_{A,k}\right)^{Q-1-j} \nonumber\\
        &\hspace{3mm} + t_{N,k}^{QR} ||\theta_k^0 -\theta^*||^2 + W_{2,k} +W_{1,k},
    \end{align*}
    where $t_{M,k}^{QR}=t_{J,k}^Q=\max \{\left( 1+t_{A,k}\right)^Q,   t_{G,k}^Q \}-1$, $W_{1,k}$ and $W_{2,k}$ are defined as in (\ref{eq:w_1k}) and (\ref{eq:w_2k}). 
    
    Next, we will check if the following four conditions hold: 
    \begin{align*} 
        &\sum_{k=0}^{\infty} t_{M,k}^{QR}<\infty, \quad 
        \lim_{k\to \infty} t_{N,k}^{QR}\leq 0, \quad
        t_C \gamma_k^2\sum_{j=0}^{Q-1} \left( 1+t_{A,k}\right)^{Q-1-j} +W_{2,k}-W_{1,k}>0, \\
        &\sum_{k=0}^{\infty} \left(t_C \gamma_k^2 \sum_{j=0}^{Q-1} \left( 1+t_{A,k}\right)^{Q-1-j}  + W_{2,k}-W_{1,k} \right) <\infty.
    \end{align*}
    
    \begin{itemize}
        \item As shown in (\ref{ieq:rk_2n_finite_1}) and (\ref{ieq:rk_2n_finite_2}), we know $\sum_{k=0}^{\infty} t_{M,k}^{QR} = \sum_{k=0}^{\infty} t_{J,k}^Q <\infty$.

        \item  
        Let us consider the term $t_{N,k}^{QR}$. As shown in the previous, there exists $K_0$ such that $t_{A,k} \leq 1$ for all $k \geq K_0$. Then for $k \geq K_0$, it has $\sum_{j=0}^{Q-1} \left( 1+t_{A,k}\right)^{Q-1-j} \leq Q2^{Q-1}$ by (\ref{ieq:t_ak_bound}) and
        \begin{align*}
         \limsup_{k\to \infty} \left(\frac{t_{N,k}^{QR}}{\beta_k}\right)  
        &\leq \frac{L_{\theta}^2 Q 2^{Q-1}}{\eta_2} \lim_{k\to \infty} \frac{ \gamma_k^{\tau}}{\beta_k} - \frac{2\mu_h R}{1+R} + \frac{R L_h^2}{1+R} \lim_{k\to \infty} \beta_k = -\frac{2\mu_h R}{1+R} <0.
    \end{align*}
    Similar to case 3, there exists $K_2$ such that for all $k \geq \max\{K_0, K_2\}$, it holds $\frac{t_{N,k}^{QR}}{\beta_k} \leq - \frac{\mu_h}{1+R}$, and we know $\beta_k >0$ for all $k$,
    \begin{align*}
        t_{N,k}^{QR} \leq - \frac{\mu_h \beta_k}{1+R} <0,
    \end{align*}
    i.e., for every $k \geq \max\{K_0, K_2\}$, we have $t_{N,k}^{QR} < 0$. 

    \item Similar to (\ref{eq:case2_t_c_gamma_k^2_1}) and (\ref{eq:case2_t_c_gamma_k^2_2}) in case $2$, it is easy to obtain 
    \begin{align*}
        t_C \gamma_k^2 \sum_{j=0}^{Q-1} \left( 1+t_{A,k}\right)^{Q-1-j} >0, \quad \text{and} \quad \sum_{k=0}^\infty \left(t_C \gamma_k^2 \sum_{j=0}^{Q-1} \left( 1+t_{A,k}\right)^{Q-1-j} \right) <\infty.
    \end{align*}
    
    From (\ref{eq:w_2k}) and (\ref{eq:w_1k}), we know 
    \begin{align*}
        W_{2,k} - W_{1,k} >0, \quad \text{and}\quad \sum_{k=0}^{\infty} \left( W_{2,k} - W_{1,k}\right) = \sum_{k=0}^{\infty} W_{2,k} < R \nu^2 \sum_{k=0}^{\infty} \beta_k^2 <\infty.
    \end{align*} 
    It is natural to have
    \begin{align*}
    &t_C \gamma_k^2\sum_{j=0}^{Q-1} \left( 1+t_{A,k}\right)^{Q-1-j} +W_{2,k}-W_{1,k}>0, \\
    \text{and} \quad  &\sum_{k=0}^{\infty} \left(t_C \gamma_k^2 \sum_{j=0}^{Q-1} \left( 1+t_{A,k}\right)^{Q-1-j}  + W_{2,k}-W_{1,k} \right) < \infty.
    \end{align*}
    \end{itemize}
    All conditions are met. Thus, we conclude that the sequence $\{x_k\}$ converges to $x^*$ almost surely for case 4 (i.e., $Q>1,R>1$) by Lemma~\ref{lemma:convergence_as_2}.

    \end{enumerate}
\end{proof}
\paragraph{Remark:} While the analysis accommodates the general case $Q>1$ and $R>1$, empirical behavior indicates that performance is often best when $Q=1$ and $R>1$. The case $Q=R=1$ is typically suboptimal, and larger values of $Q$ tend to degrade performance. These trends are illustrated in the numerical section that follows, while the theoretical guarantees hold for all admissible choices of $Q$ and $R$.

\section{Numerical studies} \label{sec:numerics}

This section is organized into two subsections, corresponding to the two classes of test problems
considered in this chapter: multi-objective optimal power flow (OPF) problems and retail
portfolio (revenue) maximization problems. For all numerical experiments, the proposed
algorithm~\ref{alg:fusion_alg} was executed under two stopping criteria: (i) a fixed total
computation time $T$, and (ii) a fixed total number of iterations $N_{\text{iters}}$. The
corresponding values were chosen based on the problem class: for the multi-objective OPF
problems, we set $N_{\text{iters}}=100$ and $T=600$ seconds, while for the retail revenue
maximization problems, we used $N_{\text{iters}}=500$ and $T=30$ seconds.

In addition, we study the influence of the following two factors on the performance of our
algorithm:
\begin{itemize}
    \item \textbf{Outer--inner iteration configuration:}  
    We vary the number of outer and inner iterations in the two-level framework, where
    $\texttt{Out}$ denotes the number of outer-level iterations and $\texttt{In}$ denotes the
    number of inner-level iterations performed at each outer iteration. In all experiments,
    we consider the following eleven $(\texttt{Out},\texttt{In})$ combinations:
    \begin{align*}
        \{ &(15,15),\; (15,7),\; (15,1),\;
           (7,15),\; (7,7),\; (7,1),\;
           (1,15),\; (1,7),\; (1,1),\; (1,0),\; (0,1) \}.
    \end{align*}
    \item \textbf{Initial step size selection:}  
    We examine the effect of the initial step size parameter $\gamma_0$ on convergence.
    Unless stated otherwise, three values are tested in the experiments,
    $\gamma_0 \in \{1.0,\,0.1,\,0.01\}$, allowing us to assess the robustness of the algorithm
    with respect to scaling.
\end{itemize}
All experiments were carried out on a desktop equipped with an NVIDIA RTX~3080 GPU and an Intel Core i5--12600K CPU. The algorithm was implemented in Python~3.8.

\subsection{Multi-objective optimal power flow (OPF)}
For the first application, we consider a multi-objective optimal power flow (OPF) problem
within the proposed two-level optimization framework. The goal is to balance two competing
objectives at the outer level: minimizing generation costs and maximizing the penetration of
renewable energy.

In this setting, the inner-level problem corresponds to a regression task that predicts solar power generation. This learning component provides the estimated renewable generation, which serves as an input to the outer-level OPF problem. Accurate solar power prediction is therefore critical, as it directly affects the amount of renewable energy available to satisfy system demand. Using these predictions, the outer-level problem solves a multi-objective OPF to determine the optimal dispatch of conventional (oil-based) generators while accounting for renewable integration.

\subsubsection{Datasets}
We construct a single OPF test instance for the Berlin region by fusing multiple real-world
data sources: (i) weather features and solar generation for the learning task, and (ii) oil
generation capacities together with total demand for the OPF constraints and objectives.
\begin{itemize}
    \item \textbf{Weather and solar (inner-level learning, 2018).}
    The solar prediction task combines feature data from the National Renewable Energy
    Laboratory (NREL)~\cite{NREL} with solar generation targets from
    Netztransparenz~\cite{Netztransparenz}. The NREL data provides $23$ structured weather
    features (e.g., temperature, humidity, wind speed), which serve as inputs to a regression
    model, while the Netztransparenz solar generation values serve as the target output.
    These data are taken over a $90$-day horizon starting from January~1,~2018, recorded at
    $15$-minute intervals (i.e., $8{,}640$ time points).

\item \textbf{Oil capacities and demand (outer-level OPF, 2025).}
The OPF layer uses real generator capacity profiles for three oil-based units together with a total electricity demand time series for Berlin, all taken from $2025$ and sourced from the SMARD platform~\cite{SMARD}. These datasets define the feasibility and power balance constraints. At each time instance, the sum of predicted solar generation and dispatched oil generation is required to meet or exceed the total demand. The resulting multi-objective OPF formulation captures the trade-off between minimizing operational cost and increasing renewable energy utilization.
\end{itemize}

\subsubsection{Results}
As described in Section~\ref{sec:intro_energy_problems}, the multi-objective OPF problem
considers two competing outer-level objectives: minimization of generation cost $f_1$ and
maximization of renewable energy penetration $f_2$. We adopt a weighted-sum scalarization
and define the outer-level optimization problem as
\begin{align*}
    \min_{x} \hspace{1mm} F_{\mathrm{ws}}(x,\theta^*) 
    = w_1 \, f_1(x,\theta^*) - w_2 \, f_2(x,\theta^*).
\end{align*}
To obtain an approximation of the Pareto front, we solve the scalarized problem for eleven
uniformly spaced weight pairs,
\begin{align*}
    (w_1,w_2) \in \left\{ (0.1k,\; 1-0.1k) \;\middle|\; k = 0,1,\ldots,10 \right\}.
\end{align*}

The inner-level problem consists of training an XGBoost regression model to predict renewable
(i.e., solar) generation. The XGBoost hyperparameters are fixed across experiments, with the maximum tree depth set to $4$ and the maximum number of leaves set to $8$. The learning rate $\beta_k$ in Algorithm~\ref{alg:fusion_alg} follows the prescribed decreasing schedule. At the inner level, updates are performed incrementally using the additive structure of XGBoost, which corresponds to a single functional gradient step at each iteration. The ensemble from iteration $k-1$ is reused and augmented with new trees at iteration $k$, rather than retraining from scratch. This continuation strategy aligns with the iterative framework and yields substantial computational savings.
We note that the proposed framework is not restricted to XGBoost and can accommodate a
broad class of machine learning models without modification. In particular, ensemble methods
and deep neural networks naturally support incremental training through the addition of
trees or layers, respectively, making them well suited to the iterative structure of the
algorithm. Simpler models, such as polynomial or linear regression, can also be incorporated within the same framework, as they admit analogous update mechanisms. 
 
Algorithmic performance is evaluated using the hypervolume (HV) indicator. For this purpose,
the final solution sets obtained from all eleven weighted-sum runs are aggregated, and the
hypervolume of the resulting combined set is computed. Prior to computing the hypervolume,
both objective functions are normalized, as generation cost and renewable penetration are
defined on different numerical scales. We employ a maximum-value normalization based on the $L_\infty$ norm. Let $S$ denote the set of all solutions obtained from the configurations under consideration (excluding the
identified outlier). For a given objective $m$, let $f_m(s)$ denote its unnormalized value for
a solution $s \in S$. The normalization factor and objective values are respectively defined as follows.
\begin{align*}
    f_m^{\max} = \max_{s' \in S} \, |f_m(s')|, \hspace{2mm}
    f_m^{\mathrm{nor}}(s) = \frac{f_m(s)}{f_m^{\max}}.
\end{align*}
This normalization preserves the relative dominance structure of the solution set, which is
critical for meaningful hypervolume comparisons. A known limitation of this normalization is its sensitivity to extreme values: a single
outlying solution can dominate the scaling and compress the remaining values. In our
experiments, the configuration with \texttt{Out}$=0$ and \texttt{In}$=1$ exhibited such
behavior and was therefore excluded prior to normalization. All subsequent hypervolume
computations are based on the remaining \texttt{Out}--\texttt{In} configurations.

Figure~\ref{fig:mop_opf_diff_out_in} shows that the hypervolume performance is governed
primarily by the number of outer-loop iterations \texttt{Out}, while the number of inner-loop
iterations \texttt{In} has a negligible or adverse effect. In
Figure~\ref{fig:mop_opf_diff_out_in}(a), the hypervolume trajectories corresponding to
different values of \texttt{In} overlap completely for any fixed \texttt{Out}. For example,
the three configurations with \texttt{Out}$=15$ and \texttt{In}$\in\{1,7,15\}$ produce
identical trajectories, as do the configurations with \texttt{Out}$=7$. This indicates that
additional inner-loop iterations do not improve solution quality once the outer-loop update
is fixed. Instead, the algorithm exhibits three distinct performance regimes determined
solely by \texttt{Out}, with \texttt{Out}$=15$ achieving the fastest convergence and highest
hypervolume, followed by \texttt{Out}$=7$ and \texttt{Out}$=1$.

The computational implications of this behavior are evident in
Figure~\ref{fig:mop_opf_diff_out_in}(b). When \texttt{Out}$= 15$, the hypervolume curves
remain overlapped under a fixed time budget, indicating that the cost of the inner loop is
negligible relative to the dominant outer-loop computation. In contrast, for the
\texttt{Out}$=1$ setting, the hypervolume achieved within a fixed time budget decreases as
\texttt{In} increases. In particular, the configuration with \texttt{In}$=0$ slightly
outperforms that with \texttt{In}$=7$, as inner-loop iterations introduce computational
overhead without yielding corresponding gains. This behavior shows that, in this application,
maximizing the frequency of outer-loop updates by minimizing the inner loop is the most
effective strategy.

\begin{figure*}[ht] 
  \centering
  \begin{subfigure}{.45\textwidth}
    \centering
    \includegraphics[width=\linewidth]{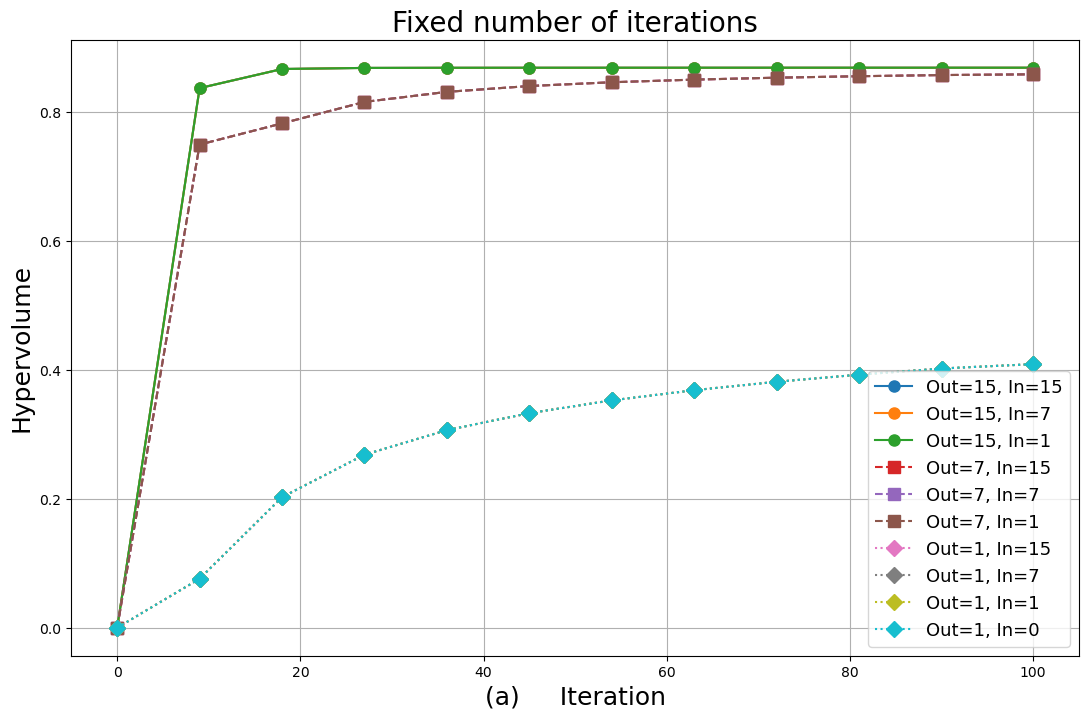}
  \end{subfigure}%
\begin{subfigure}{.45\textwidth}
    \centering
    \includegraphics[width=\linewidth]{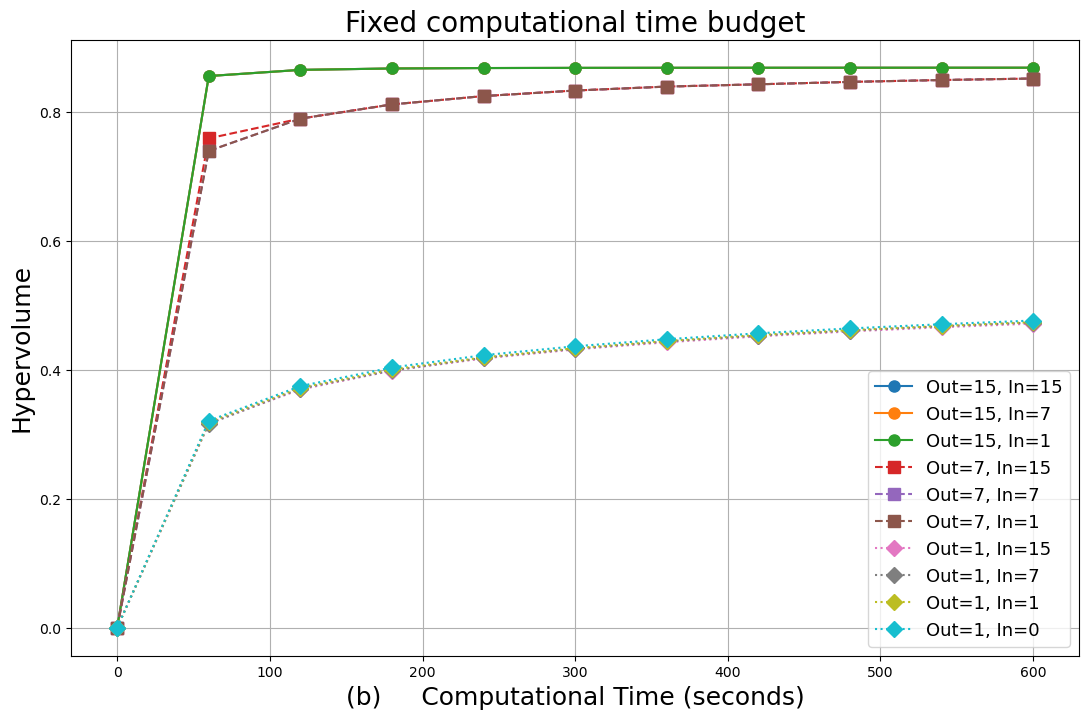}
  \end{subfigure} 
  \caption{Hypervolume convergence for different outer- and inner-loop configurations.
For any fixed \texttt{Out}, the trajectories overlap in panel~(a); the same effect appears
in panel~(b) for \texttt{Out}$=15$ and \texttt{Out}$=1$.
}
\label{fig:mop_opf_diff_out_in}
\end{figure*}

\begin{figure*}[ht] 
  \centering
  \begin{subfigure}{.45\textwidth}
    \centering
    \includegraphics[width=\linewidth]{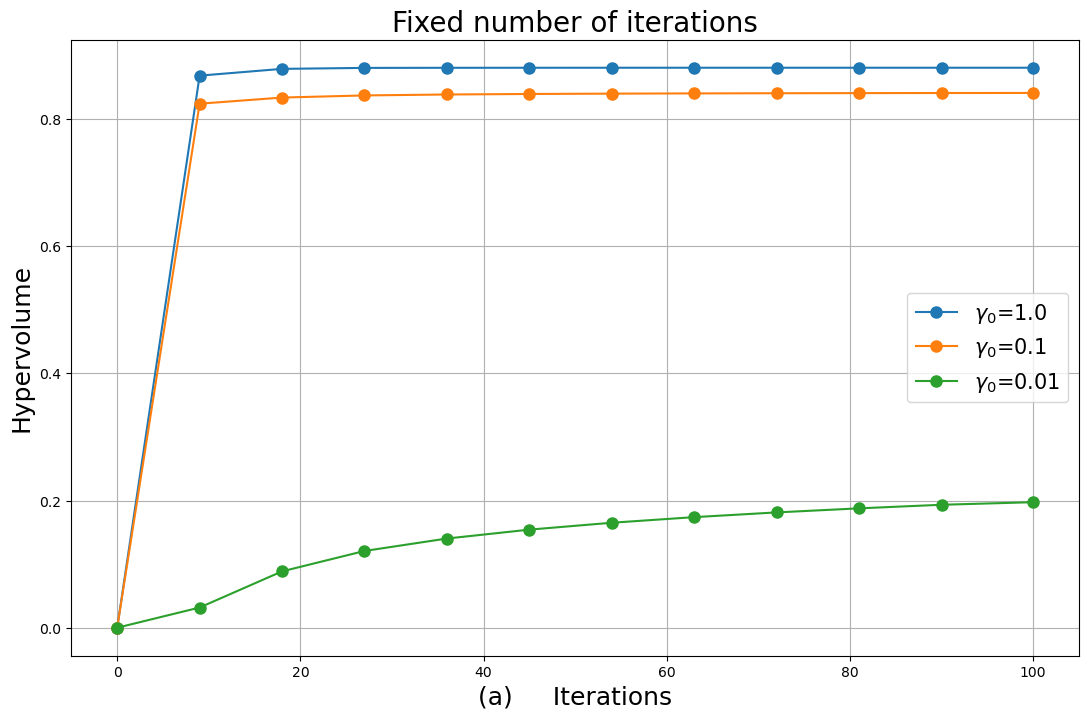}
  \end{subfigure}%
\begin{subfigure}{.45\textwidth}
    \centering
    \includegraphics[width=\linewidth]{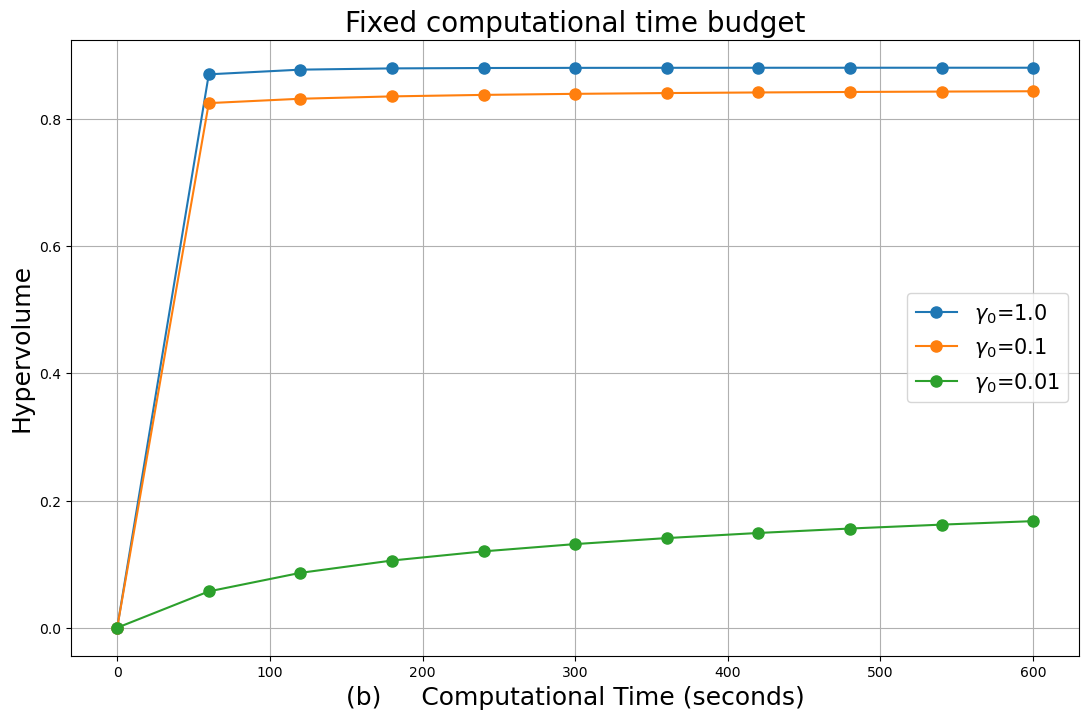}
  \end{subfigure} 
  \caption{Hypervolume convergence for \texttt{Out}$=15$ and \texttt{In}$=15$ under different
initial stepsizes $\gamma_0$, using~(a) fixed iteration $N_{\text{iters}}=100$ and~(b) fixed time budgets $T=600$.
}
  \label{fig:mop_opf_diff_stepsize}
\end{figure*}

As shown in Figure~\ref{fig:mop_opf_diff_stepsize}, the choice of the initial step size $\gamma_0$ is another important factor. It not only determines the convergence rate but also the quality of the final solution. Results based on iterations (Figure~\ref{fig:mop_opf_diff_stepsize}a) and time (Figure~\ref{fig:mop_opf_diff_stepsize}b) are consistent, showing that an aggressive step size of $\gamma_0 =1.0$ yields superior performance. This setting allows the algorithm to converge rapidly. It achieves a hypervolume of approximately $0.85$ after just $20$ iterations, equivalent to roughly $100$ seconds of computational time. By contrast, a moderately smaller step size of $\gamma_0=0.1$ converges to a suboptimal plateau with a hypervolume of approximately $0.81$. Furthermore, a conservative step size of $\gamma_0=0.01$ leads to slow progress. Its hypervolume reaches only about $0.2$ in the end, which is a classic case to show that a too small learning rate leads to computational inefficiency. 

\subsection{Portfolio maximization problems}

\subsubsection{Datasets}

We consider two datasets for the retail portfolio optimization problems: the \emph{Logit
dataset} and \emph{Cohen’s dataset}. Both datasets consist of historical price–sales
observations and are commonly used in demand modeling and pricing studies. The Logit dataset from is a synthetic dataset generated using the
classical logit demand model defined in~(\ref{eq:demand_function}). It contains data for
$50$ products observed over a horizon of $50$ weeks, resulting in a total of $2{,}500$
price–sales entries. This dataset has been used as a benchmark for testing portfolio optimization and pricing algorithms
(e.g.,~\cite{kannan2023solving,logit23}).
Cohen’s dataset is an anonymized real-world dataset obtained from a US-based retailer,
originally introduced in~\cite{cohen2022demand}. It comprises $44$ products (SKUs) with
historical weekly prices and sales spanning the period from $2016$ to $2018$, covering
approximately $98$ weeks. This dataset has been extensively used in the literature for
demand estimation and pricing research (see~\cite{cohen2022data, deng2022unified,
kannan2023solving}). Together, these two datasets enable a systematic evaluation of the proposed MSLO method.
The Logit dataset provides a controlled setting with an interpretable demand structure,
while Cohen’s dataset captures the variability and noise inherent in real retail data.

\subsubsection{Results}
Figure~\ref{fig:logit_diff_out_in}(a) shows that performance is driven mainly by the frequency of outer updates. Increasing the number of outer iterations leads to significant revenue improvements, whereas using very few outer updates results in poor performance regardless of inner accuracy. Moreover, running too many inner iterations can degrade the final solution. For a fixed \texttt{Out}, the more inexact settings \texttt{In}$=1$ and \texttt{In}$=7$ consistently outperform the heavier \texttt{In}$=15$ configuration. This indicates that over-solving the inner problem can slow global progress and may introduce overfitting effects.


Furthermore, Figure~\ref{fig:logit_diff_out_in}(b) more clearly demonstrates the practical advantages of inexact learning by taking into account computational cost. In this time-based comparison, the \texttt{Out}$=15$, \texttt{In}$=1$ configuration clearly performs best, converging faster than all other settings and achieving the highest total revenue of approximately $51$ within $5$ seconds. By contrast, the \texttt{Out}$=15$, \texttt{In}$=15$ configuration is the slowest in its \texttt{Out}$=15$ group, requiring the entire $30$ seconds budget just to approach the performance that the inexact setting achieved in a fraction of the time.
This indicates that the additional inner-loop computations incur a substantial computational cost that is not offset by their relatively small per-iteration performance gains. Consequently, these results favor an inexact strategy for the logit problem: minimizing inner-loop work to \texttt{In}$=1$ allows more frequent outer-loop updates, which leads to faster convergence and higher overall revenue under a fixed time budget.

\begin{figure*}[ht] 
  \centering
  \begin{subfigure}{.45\textwidth}
    \centering
    \includegraphics[width=\linewidth]{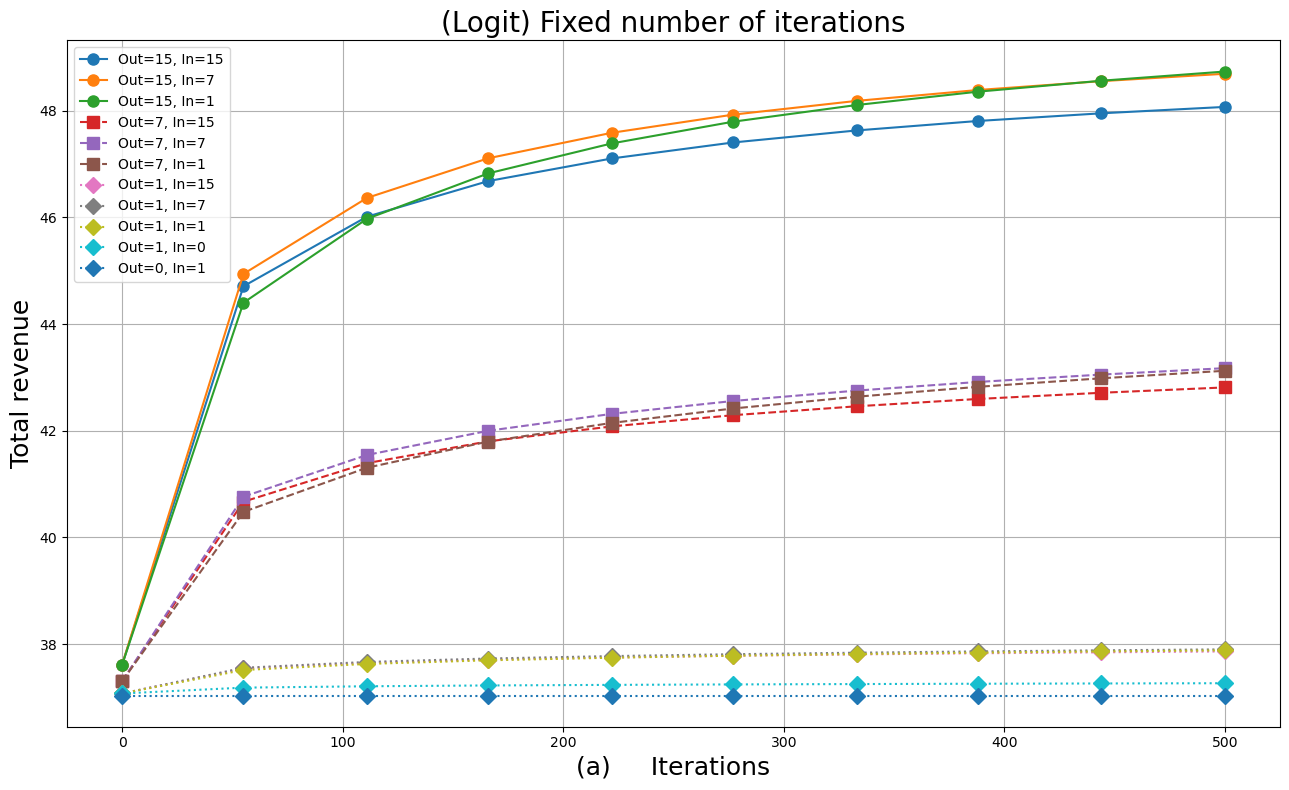}
  \end{subfigure}%
\begin{subfigure}{.45\textwidth}
    \centering
    \includegraphics[width=\linewidth]{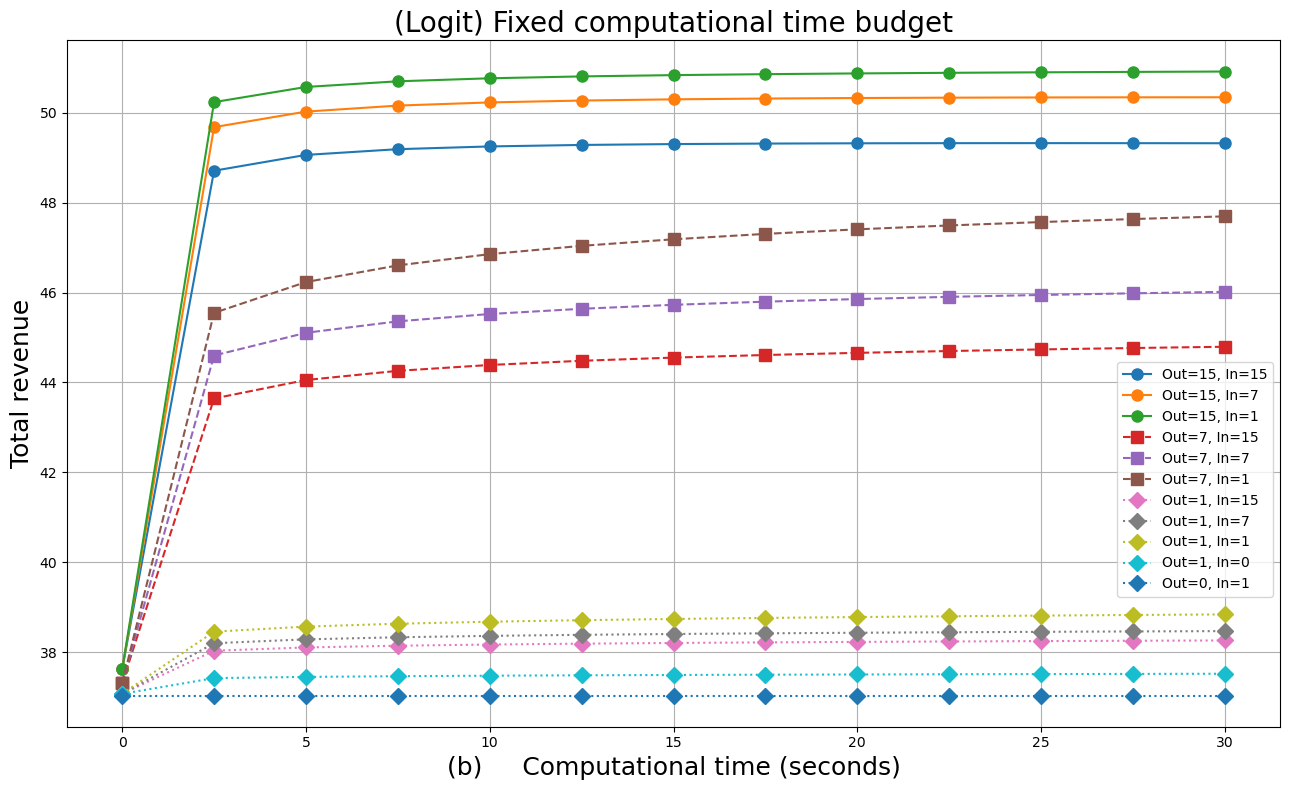}
  \end{subfigure} 
  \caption{Total revenue for different combinations of outer loop (\texttt{Out}) and inner loop (\texttt{In}) counts for the Logit dataset. (a) Performance under a fixed number of iterations $N_{\text{iters}}=500$. 
  (b) Performance under a fixed computational time budget $T=30$. 
  }
  \label{fig:logit_diff_out_in}
\end{figure*}

\begin{figure*}[ht] 
  \centering
  \begin{subfigure}{.45\textwidth}
    \centering
    \includegraphics[width=\linewidth]{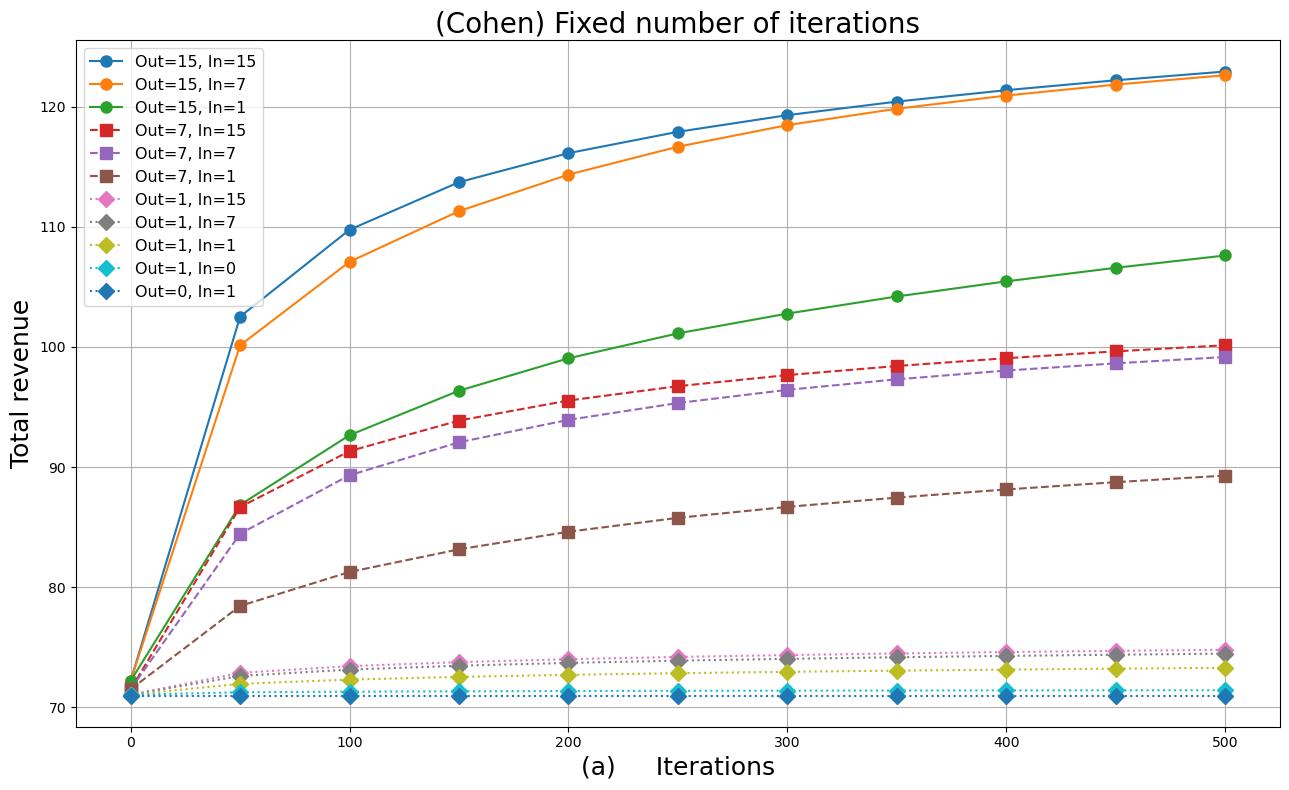}
  \end{subfigure}%
\begin{subfigure}{.45\textwidth}
    \centering
    \includegraphics[width=\linewidth]{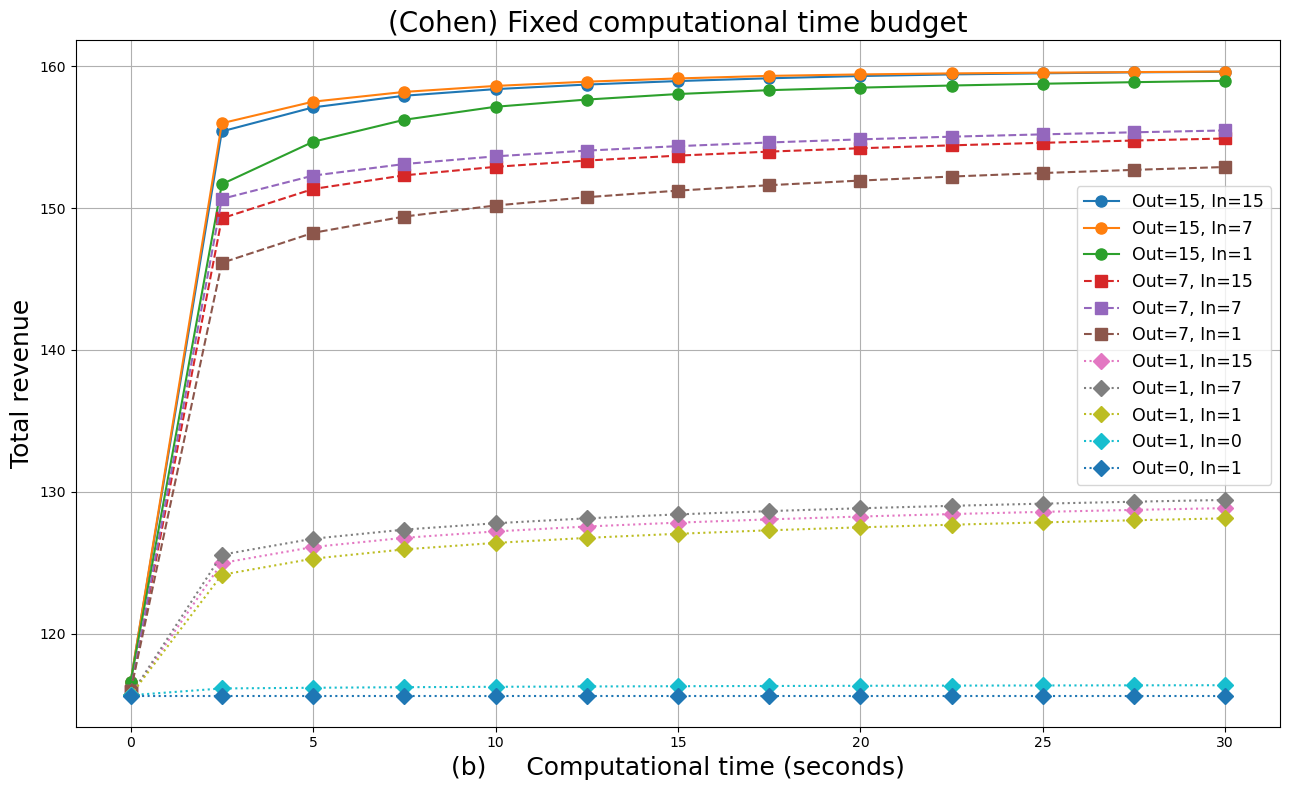}
  \end{subfigure} 
  \caption{Total revenue for different combinations of outer loop (\texttt{Out}) and inner loop (\texttt{In}) counts for the Cohen dataset. (a) Performance under a fixed number of iterations $N_{\text{iters}}=500$.
  (b) Performance under a fixed computational time budget $T=30$. 
  }
  \label{fig:cohen_diff_out_in}
\end{figure*}

The results for the Cohen dataset (Figure~\ref{fig:cohen_diff_out_in}) indicate that, in contrast to the logit case, a more exact inner-loop solution is the more efficient strategy. The iteration-based analysis in Figure~\ref{fig:cohen_diff_out_in}(a) shows a clear and consistent benefit from increasing the number of inner-loop iterations. For any fixed outer-loop count, higher inner-loop counts lead to greater revenue per iteration, indicating that additional inner-loop work produces higher-quality outer updates. The time-based analysis in Figure~\ref{fig:cohen_diff_out_in}(b) further confirms that this added inner-loop effort is computationally worthwhile. The \texttt{Out}$=15$, \texttt{In}$=15$ and \texttt{Out}$=15$, \texttt{In}$=7$ configurations achieve the highest revenue and converge the fastest within the available time budget. In contrast, the inexact \texttt{Out}$=15$, \texttt{In}$=1$ configuration, which appeared less competitive in the iteration-based analysis, is also suboptimal in wall-clock time, converging at a slightly slower rate and to a lower revenue level. This behavior suggests that, for the Cohen problem, the inner loop is computationally efficient: the performance gains obtained from a more accurate inner-loop solution outweigh its additional time cost. 

However, increasing the number of inner iterations beyond a moderate level does not continue to improve performance. In particular, using \texttt{In}$=7$ always slightly outperforms using \texttt{In}$=15$. This suggests that, under a fixed computational budget, investing excessive effort in solving the inner problem very accurately can be counterproductive, as the resulting solution is used only briefly before the outer variables change again. Overall, these results highlight that the optimal tradeoff between inner- and outer-loop computation is problem-dependent.

Figures~\ref{fig:logit_stepsize} and~\ref{fig:cohen_stepsize} examine the effect of the initial step size $\gamma_0$ on revenue for the logit and Cohen datasets under both fixed iteration and fixed computational time budgets. Across both datasets, the choice of $\gamma_0$ plays a critical role in determining convergence speed as well as the final achieved performance. Large step sizes ($\gamma_0 \geq 1.0$) consistently enable rapid progress toward high-revenue solutions. For the logit dataset, a step size of $\gamma_0 = 1.0$ is sufficient to reach the optimal revenue within a small number of iterations. A moderate step size of $\gamma_0 = 0.1$ produces steady improvement, but the rate of progress is too slow to approach the optimum within either the iteration limit or the available computational time. With a smaller step size of $\gamma_0 = 0.01$, the algorithm effectively stagnates, yielding only a marginal improvement over the baseline solution.

This sensitivity to the step size is even more pronounced for the Cohen dataset. When $\gamma_0 = 10.0$, the algorithm reaches the optimal revenue almost immediately. A smaller but still aggressive choice of $\gamma_0 = 1.0$ converges reliably, but requires nearly the entire computational budget to achieve comparable performance. In contrast, smaller step sizes again fail to produce meaningful gains within the allowed time. A consistent pattern therefore emerges across both datasets: when iteration budgets are limited, conservative step sizes restrict exploration and prevent the method from reaching high-revenue regions. When time is limited, additional runtime cannot compensate for an overly cautious initialization. These results underscore the importance of sufficiently large initial step sizes for efficient convergence, and suggest that larger or more complex problem instances, such as the Cohen dataset, may require more aggressive step-size choices.

\begin{figure*}[ht] 
  \centering
  \begin{subfigure}{.45\textwidth}
    \centering
    \includegraphics[width=\linewidth]{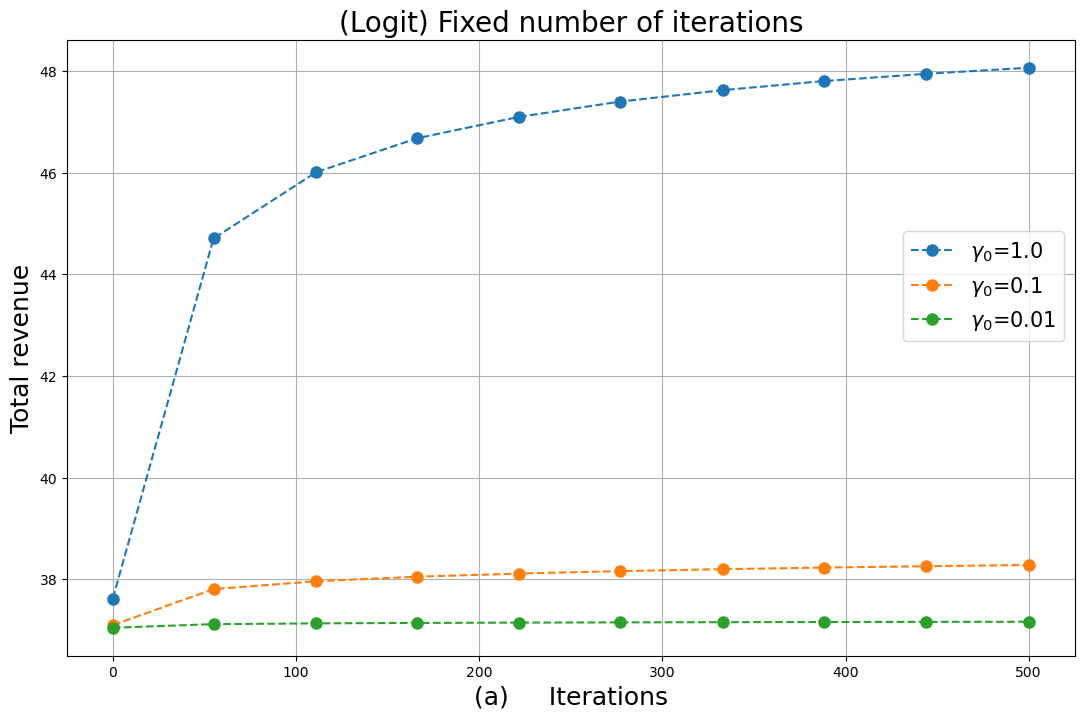}
  \end{subfigure}%
\begin{subfigure}{.45\textwidth}
    \centering
    \includegraphics[width=\linewidth]{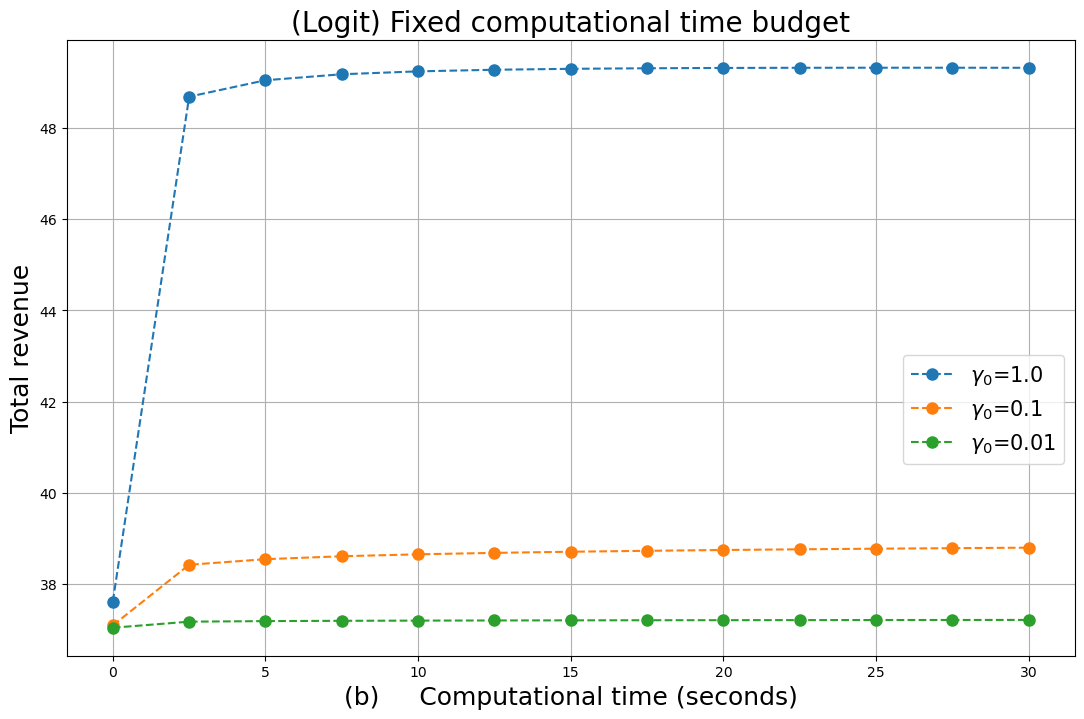}
  \end{subfigure} 
  \caption{Total revenue with a fixed configuration of \texttt{Out}$=15$ and \texttt{In}$=15$ under different initial stepsizes for the Logit dataset, under~(a) fixed iterations $N_{\text{iters}}=500$ and~(b) computational time budgets $T=30$.}
  \label{fig:logit_stepsize}
\end{figure*}

\begin{figure*}[ht] 
  \centering
  \begin{subfigure}{.45\textwidth}
    \centering
    \includegraphics[width=\linewidth]{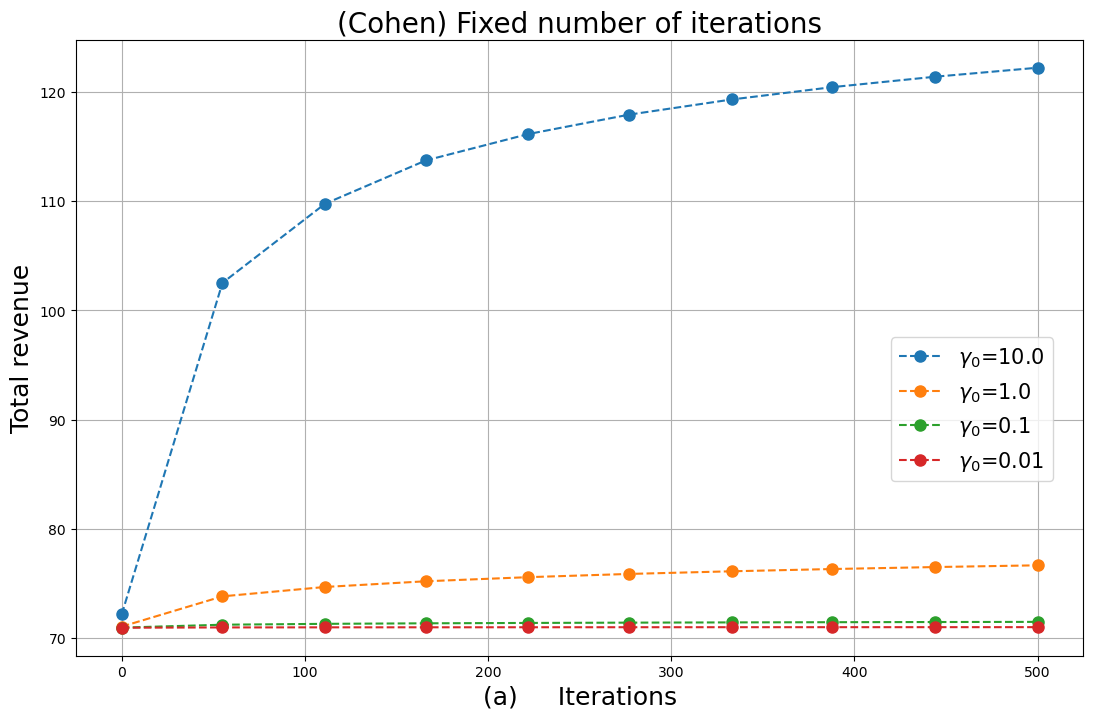}
  \end{subfigure}%
\begin{subfigure}{.45\textwidth}
    \centering
    \includegraphics[width=\linewidth]{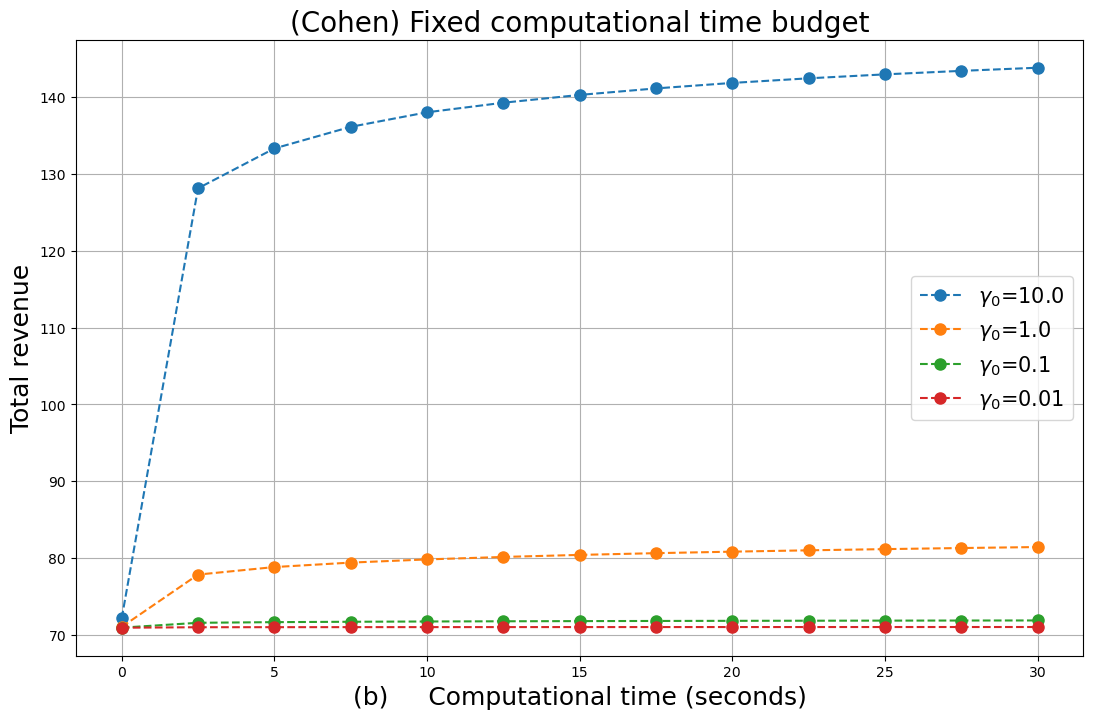}
  \end{subfigure} 
  \caption{Total revenue with a fixed configuration of \texttt{Out}$=15$ and \texttt{In}$=15$ under different initial stepsizes for the Cohen dataset, under~(a) fixed iterations $N_{\text{iters}}=500$ and~(b) computational time budgets $T=30$.}
  \label{fig:cohen_stepsize}
\end{figure*}
\section{Conclusion} \label{sec:conclusion}
This paper studies a class of two-level optimization problems in which an outer-level optimization problem and an inner-level learning task are tightly intertwined. The outer and inner problems are not separable, but instead interact through iterative updates, making the overall problem structure inherently joint. Such formulations naturally arise in applications such as energy management and retail optimization. 

Our primary contribution is the proposal of a learning--optimization framework in which the outer-level decision variables and the inner-level model parameters are updated jointly in an iterative manner. Within this framework, we introduce both single-step and multi-step learning--optimization schemes, which differ in the amount of inner-level work performed between successive outer-level updates. These schemes provide a flexible mechanism to balance solution quality and computational cost. We establish theoretical convergence guarantees for the proposed framework under mild and interpretable conditions that are satisfied by the problem class studied in this work, in which the inner-level learning objective is convex and the outer-level optimization objective is pseudoconvex. In addition, we allow for controlled inexactness in the inner-level solutions and characterize how the resulting errors propagate through the outer-level updates. The analysis also accommodates stochastic approximation–type algorithms, where updates are subject to stochastic noise satisfying standard boundedness and summability conditions (in both the outer and inner levels). Under these conditions, we show that the proposed algorithms converge to a stationary solution of the joint problem. The practical performance of the proposed methods is evaluated through numerical experiments on real-world datasets, covering both single-objective and multi-objective settings. The experiments reveal two consistent empirical insights. First, solution quality is primarily driven by the number of outer-level updates rather than by the accuracy of the inner-level solves. In particular, for the multi-objective OPF problem, performing more than one inner iteration provides negligible benefit, and time-based results confirm that minimizing inner-loop effort leads to faster convergence by enabling more frequent outer-loop updates. Second, the choice of the initial step size plays a critical role in convergence behavior. Conservative step sizes consistently lead to stagnation at suboptimal solutions, whereas sufficiently aggressive step sizes are essential for achieving rapid convergence and high-quality solutions.

Several directions for future research emerge from this work. From a theoretical standpoint, it would be valuable to investigate whether the convexity requirement on the inner-level problem can be relaxed to pseudoconvexity, thereby broadening the applicability of the framework. Another promising direction is to analyze robustness under misspecified or uncertain feasible sets in the outer-level problem, extending beyond objective uncertainty. This setting naturally connects to pseudomonotone variational inequalities in areas such as games and structural design, where such joint frameworks are highly relevant. On the algorithmic side, adaptive and dynamic loop-scheduling strategies merit further exploration, such as allocating more inner-loop effort in early iterations or terminating inner solves based on problem-dependent criteria. Finally, a formal convergence-rate analysis would provide deeper insight into the efficiency of the proposed methods and offer principled guidance for parameter selection.

\section*{Declarations}

\subsection*{Ethics approval and consent to participate}
Not applicable. The study does not involve human participants, animals, or any data requiring ethical approval.

\subsection*{Consent for publication}
Not applicable. No individual-level personal data or identifiable information is included in this manuscript.

\subsection*{Availability of data and material}
The datasets used and/or analysed during the current study are publicly available (as cited within the manuscript) or are available from the corresponding author upon reasonable request.

\subsection*{Competing interests}
The authors declare that they have no competing interests.

\subsection*{Funding}
This research was funded partly by the Deutsche Forschungsgemeinschaft (DFG, German Research Foundation) under Germany's Excellence Strategy --- The Berlin Mathematics Research Center MATH+ (EXC-2046/1, project ID: 390685689). No additional external funding was received.

\subsection*{Authors' contributions}
Zijun Li carried out the computational work, conceived the algorithmic framework, implemented the algorithmic framework from scratch, improved the algorithm, and contributed to writing the manuscript. 
Aswin Kannan conceived the study, conceived the algorithmic framework, supervised the research, developed the methodological direction, and contributed to writing and revising the manuscript. Both authors read and approved the final manuscript.

\subsection*{Acknowledgements}
The authors thank Humboldt Universität zu Berlin and the International Institute of Information Technology Bangalore for their support during the development of this work. 

\bibliographystyle{plain}
\bibliography{bibliography,sample}
\end{document}